\documentclass[12pt]{amsproc}
\usepackage[english]{babel}
\usepackage{paralist,url,verbatim, anysize}
\usepackage{amscd}
\usepackage[all,cmtip]{xy} 
\usepackage{amsmath,amsthm,amssymb,amsfonts}
\usepackage{mathtools}
\usepackage{leftidx}
\usepackage[usenames, dvipsnames]{color}
\usepackage{graphicx}
\usepackage{euler}
\usepackage{MnSymbol}
\usepackage[bookmarks=false]{hyperref}
\usepackage{gfsartemisia-euler}
\usepackage[utf8]{inputenc}

\makeatletter
\let\@fnsymbol\@arabic
\makeatother

\theoremstyle{plain}
\newtheorem{theorem}{Theorem}
\newtheorem{proposition}{Proposition}
\newtheorem{notation}{Notation}

\newtheorem{lemma}{Lemma}
\newtheorem{remark}{Remark}

\newtheorem{definition}{Definition}
\newtheorem{example}{Example}

\def\I{\mathcal{I}}
\def\J{\mathcal{J}}

\def\p{\mathfrak{p}}

\DeclareMathOperator{\Gr}{Gr}
\DeclareMathOperator{\M}{M}

\DeclareMathOperator{\im}{im}
\DeclareMathOperator{\Span}{Span}

\DeclareMathOperator{\Proj}{Proj}
\DeclareMathOperator{\Spec}{Spec}

\DeclareMathOperator{\rank}{rank}

\DeclareMathOperator{\height}{height}

\title{Results on the algebraic matroid of the determinantal variety}

\author{Manolis C. Tsakiris}
\address{Key Laboratory of Mathematics Mechanization, Academy of Mathematics and Systems Science, Chinese Academy of Sciences, Beijing, 100190, China}
\email{manolis@amss.ac.cn}
\address{Dipartimento di Matematica,  Universit\`a di Genova, Via Dodecaneso 35, 16146 Genova, Italy}

\begin{document}

\begin{abstract}
We make progress towards characterizing the algebraic matroid of the determinantal variety defined by the minors of fixed size of a matrix of variables. Our main result is a novel family of base sets of the matroid, which characterizes the matroid in special cases. Our approach relies on the combinatorial notion of relaxed supports of linkage matching fields that we introduce, our interpretation of the problem of completing a matrix of bounded rank from a subset of its entries as a linear section problem on the Grassmannian, and a connection that we draw with a class of local coordinates on the Grassmannian described by Sturmfels $\&$ Zelevinsky in 1993. 
\end{abstract}

\maketitle

%%%%%%
\section{Introduction} \label{section:Introduction}

\subsection{The determinantal variety} Determinantal varieties and ideals play an important role in algebraic geometry and commutative algebra and interact strongly with combinatorics. Since Macaulay \cite{macaulay1916algebraic}, they have been the objects of study by many researchers, e.g.,\cite{eagon1962ideals,hochster1971cohen,lascoux1978syzygies,merle1981sections,huneke1983strongly,eisenbud1988linear,sturmfels1990grobner,bernstein1993combinatorics,bruns2003grobner,boocher2012free,raicu2014local,conca2018hankel,conca2019lovasz}. The simplest and perhaps most important determinantal variety, which we will denote by $\M(r,m \times n)$, is defined over a field $k$ by the vanishing of the $(r+1)$-minors of an $m \times n$ matrix of variables. Its $k$-rational points are the $m \times n$ matrices with entries in $k$ of rank at most $r$. The variety $\M(r,m \times n)$ has been studied extensively and many good properties have been established \cite{BrunsVetter:1988}. For example, Hochster \& Eagon proved that it is Cohen-Macaulay \cite{hochster1971cohen}, Lascoux \cite{lascoux1978syzygies} gave in characteristic zero a minimal free resolution of its defining ideal, and Merle \& Giusti \cite{merle1981sections}, Eisenbud \cite{eisenbud1988linear} and Conca \& Welker \cite{conca2019lovasz} studied classes of linear sections of the variety. On the other hand, the important question of which coordinate projections from the variety have finite generic fiber, remains largely unanswered. This is tantamount to characterizing the algebraic matroid associated to $\M(r,m \times n)$, a problem for which little is known. This question is also of great significance for the machine learning problem of low-rank matrix completion, where one aims to complete a rank-$r$ matrix from its partially observed entries \cite{kiraly2015algebraic}. In this paper, we make progress in this direction.

We set up some notations. We will assume that $k$ is an infinite field and whenever $s$ is a positive integer we set $[s]=\{1,\dots,s\}$. We let $Z = \big(z_{ij}: \, (i,j) \in [m] \times [n]\big)$ be a matrix of variables $z_{ij}$, and $k[Z]=k[z_{ij}: \, (i,j) \in [m] \times [n]]$ the polynomial ring generated by these variables over $k$. We consider the determinantal ideal $I_{r+1}(Z)$ generated by all $(r+1)\times(r+1)$ minors of the matrix $Z$. This ideal is known to be prime of height $(m-r)(n-r)$ \cite{BrunsVetter:1988}. It thus defines the integral determinantal $k$-scheme $\M(r,m \times n) = \Spec k[Z]/I_{r+1}(Z)$ of dimension $r(m+n-r)$. 

\begin{notation} \label{not:rmn}
Throughout we will assume that $0<r<m \le n$.
\end{notation} 

%%%%
\subsection{The algebraic matroid}  \label{subsection:TheAlgebraicMatroid} Attached to the variety $\M(r,m \times n)$ is a matroid, which we now describe from an algebraic and geometric point of view; see \cite{rosen2020algebraic} for a beautiful exposition to algebraic matroids with applications. The ground set of the matroid consists of the indices $[m] \times [n]$ of the matrix entries. The independent sets of the matroid are those subsets $\Omega \subset [m] \times [n]$ for which the images in $k[Z]/I_{r+1}(Z)$ of the $z_{ij}$'s with $(i,j) \in \Omega$ are algebraically independent over $k$. Set $k[Z_\Omega]=k[z_{ij}: \, (i,j) \in \Omega]$. Then $\Omega$ is an independent set if and only if $k[Z_\Omega] \cap I_{r+1}(Z) = 0$. Set $\mathbb{A}^{\Omega} = \Spec k[Z_\Omega]$ the affine space of dimension $\#\Omega$ associated to the matrix entries indexed by $\Omega$; the $k$-rational points of $\mathbb{A}^{\Omega}$ can be thought of as $m \times n$ matrices over $k$ with zero entries outside $\Omega$. Let $\pi_\Omega: \M(r,m \times n) \rightarrow \mathbb{A}^{\Omega}$ be the projection morphism induced by the ring homomorphism $\varphi_\Omega: k[Z_\Omega] \rightarrow k[Z]/I_{r+1}(Z)$, where $z_{ij} \mapsto z_{ij}+I_{r+1}(Z)$ for $(i,j) \in \Omega$. On the level of $k$-rational points, $\pi_\Omega$ places zeros at the entries of a matrix outside $\Omega$. The scheme-theoretic image of $\pi_\Omega$, that is the closure of the image of $\pi_\Omega$, is given by $\Spec k[Z_\Omega] / k[Z_\Omega] \cap I_{r+1}(Z)$. Hence $\Omega$ is independent if and only if $\varphi_\Omega$ is injective if and only if $\pi_\Omega$ is dominant. 

To understand the independent sets, it is enough to understand the so-called bases of the matroid, which are the maximal independent sets. By a fundamental theorem of matroid theory, every base set has the same size, called the rank of the matroid. Here, this is equal to $\dim\M(r,m \times n) = r(m+n-r)$, the transcendence degree over $k$ of the field of fractions of the integral domain $k[Z]/I_{r+1}(Z)$. Thus $\Omega$ is a base set if and only if $\#\Omega = r(m+n-r)$ and $\pi_\Omega$ is dominant. By the upper-semicontinuity of the dimension of fibers \cite{hartshorne1977algebraic,Eisenbud-2004,Vakil-AG}, this is equivalent to $\pi_\Omega$ having finite fiber over $\pi_\Omega(X)$, for every $X$ in a dense open set of $\M(r,m \times n)$.

\begin{notation}
Properties of the matroid are sometimes conveniently expressed in terms of the bipartite graph $G_\Omega$ induced by 
$\Omega \subset [m] \times [n]$. The two parts of vertices are $[m]$ and $[n]$ and there is an edge between $i \in [m]$ and $j \in [n]$ if and only if $(i,j) \in \Omega$. 
\end{notation}

%%%%%%
\subsection{Known facts} \label{subsection:known-facts} Little is known from the existing literature about the algebraic matroid of $\M(r,m \times n)$. We give a brief summary of the main known facts. 

By what we said above, $\Omega$ is independent if and only if $k[Z_\Omega] \cap I_{r+1}(Z)=0$. As an ideal of $k[Z_\Omega]$, $k[Z_\Omega] \cap I_{r+1}(Z)$ can be understood via elimination theory \cite{Eisenbud-2004}. Indeed, it is generated by the members of a lexicographic Gr{\"o}bner basis of $I_{r+1}(Z)$ that lie in $k[Z_\Omega]$, where the variables $Z_\Omega$ are the least significant. Theoretically determining such a basis is in general very difficult. The following is an exception. 

For extreme ranks $r=1$ and $r=m-1$, where $I_{r+1}(Z)$ is generated by $2$-minors and maximal minors respectively, a universal Gr{\"o}bner basis is known. This is a set of generators for $I_{r+1}(Z)$ which is a Gr{\"o}bner basis under any term order, and thus it is a Gr{\"o}bner basis for the aforementioned lexicographic order. For $r=1$, the independent work of Sturmfels \cite{sturmfels1996grobner} and Villarreal \cite{villarreal2001monomial} implies the existence of a universal Gr{\"o}bner basis supported on the cycles of the complete bipartite graph $K_{m,n}$, e.g., see Theorem 3.1 in \cite{conca2007linear}. This makes the base sets for $r=1$ those $\Omega$'s of size $m+n-1$, for which $G_\Omega$ is a tree; Cucuringu \& Singer arrived at the same conclusion using rigidity theory \cite{singer2010uniqueness}. For $r=m-1$, Bernstein \& Zelevinsky \cite{bernstein1993combinatorics} proved that the maximal minors form a universal Gr{\"o}bner basis, a result that was later generalized in \cite{boocher2012free,conca2015universal,conca2020universal}. This makes the base sets for $r=m-1$ those $\Omega$'s of size $(m-1)(n+1)$, for which $G_\Omega$ does not contain the complete bipartite graph $K_{m,m}$. 

For $1<r\le m-2$ it is known that the $(r+1)$-minors are not a universal Gr{\"o}bner basis for $I_{r+1}(Z)$. Instead, they are a Gr{{\"o}}bner basis under any so-called diagonal or anti-diagonal term order, e.g., see \cite{bruns2022determinants} for a detailed treatment and also \cite{narasimhan1986irreducibility} and \cite{sturmfels1990grobner}. As noted by Kalkbrener $\&$ Sturmfels \cite{kalkbrener1995initial}, this yields a family of base sets for any $r$: with the partial order $(i,j) \le (i',j')$ if $i \le i'$ and $j \le j'$, an $\Omega$ is a base set if it does not contain an antichain of cardinality $r+1$, that is if $\Omega$ contains at most $r$ incomparable elements.  Already though for $r=1$, it is easy to find examples of base sets that do not satisfy this condition. Thus, this family does not characterize the entire algebraic matroid. Instead, Kalkbrener $\&$ Sturmfels proved the remarkable fact that one obtains the entire matroid of $\M(r, m \times n)$ as the union of the independence complexes of the initial ideals of $I_{r+1}(Z)$ corresponding to all lexicographic orders. On the other hand, getting a handle on the initial ideal usually requires a handle on the corresponding Gr{\"o}bner basis, and so this route appears to be equally difficult as determining $k[Z_\Omega] \cap I_{r+1}(Z)$ via elimination for all $\Omega$'s.

For $r=2$, D.I. Bernstein took a different route. He used the tropicalization of the Grassmannian $\Gr(2,m)$ given by Speyer \& Sturmfels \cite{speyer2004tropical} together with a connection with the completion of tree metrics. He then characterized the bases as those $\Omega$'s of size $2(m+n-2)$, for which $G_\Omega$ admits an acyclic orientation with no alternating trails; see \cite{bernstein2017completion} for further details. However, the tropicalization of the Grassmannian  becomes too complicated for $r \ge 3$, and so this approach appears to be difficult to apply for higher rank values. 

\subsection{Results of this paper} We call a set $\Phi = \bigcup_{j \in [m-r]} \phi_j \times \{j\} \subset [m] \times [m-r]$ an $(r,m)$-SLMF (Support of a Linkage Matching Field), if $\Phi$ satisfies the conditions 
\begin{align} 
\#\phi_j = r+1, \,  j \in [m-r] \, \, \, \, \, \, \, \, \, \text{and} \, \, \, \, \, \, \, \, \,  \# \bigcup_{j \in \J} \phi_j \ge \# \J + r, \, \, \,  \emptyset \neq \J \subseteq [m-r]. \label{eq:SLMF}
\end{align} SLMF's arise as the supports of the vertices of the Newton polytope of the product of maximal minors of an $m \times (m-r)$ matrix of variables. These were introduced by Sturmfels $\&$ Zelevinsky \cite{sturmfels1993maximal} in their effort to establish the aforementioned universal Gr{{\"o}}bner basis property of maximal minors, and have recently found applications in tropical geometry, e.g., by Fink \& Rinc\'on \cite{fink2015stiefel} and Loho \& Smith \cite{loho2020matching}. 

Here we introduce a generalization of the notion of SLMF:
\begin{notation} \label{not:Omega}
We write $\Omega = \bigcup_{j \in [n]} \omega_j \times \{j\}$ with the $\omega_j$'s subsets of $[m]$ and $\Omega_\J =  \bigcup_{j \in \J} \omega_j \times \{j\}$ whenever $\J \subset [n]$. 
\end{notation}

\begin{definition} \label{dfn:relaxedSLMF}
For $\J \subset [n]$ and $\nu$ a positive integer, we call $\Omega_\J$ a relaxed $(\nu,r,m)$-SLMF, if $\sum_{j \in \J} \max\big\{\# (\omega_j \cap \I) -r, 0 \big\} \le \nu (\# \I -r)$ for every $\I \subset [m]$ with $\# \I \ge r+1$, and equality holds for $\I=[m]$. 
\end{definition}

\noindent An $(r,m)$-SLMF is always a relaxed $(1,r,m)$-SLMF (Lemma \ref{lem:SLMF-I}), while from a relaxed $(1,r,m)$-SLMF we can always construct an $(r,m)$-SLMF (Lemma \ref{lem:SLMF-induced}). 

In \cite{sturmfels1993maximal} Sturmfels $\&$ Zelevinsky showed that a family of local coordinates on the Grassmannian $\Gr(r,m)$, already known by Gelfand, Graev $\&$ Retakh \cite{gelfand1990gamma} from an analytic point of view, could be seen as induced by $(r,m)$-SLMF's (section \ref{section:SLMF}). This connection is a key ingredient behind the main result of this paper: 

\begin{theorem} \label{thm:main}
If $\#\Omega = r(m+n-r)$ and there is a partition $[n] = \bigcup_{\ell \in [r]} \J_\ell$ with $\Omega_{\J_\ell}$ a relaxed $(1,r,m)$-SLMF $\forall \ell \in [r]$, then $\Omega$ is a base of the algebraic matroid of $\M(r, m \times n)$. 
\end{theorem}

Theorem \ref{thm:main} is illustrated in Examples \ref{ex:main-multiplePoints}, \ref{ex:main} and \ref{ex:flexibility} in section \ref{section:Examples}. Another key ingredient behind Theorem \ref{thm:main} is a novel interpretation of the problem of completing a matrix $X \in k^{m \times n}$ of rank at most $r$ from a subset of its entries, as a linear section problem on the Grassmannian $\Gr(r,m)$ via Pl{\"u}cker coordinates. A natural consequence is:   

\begin{proposition} \label{prp:necessary}
If $\Omega$ is a base set of the algebraic matroid of $\M(r, m \times n)$, then $\Omega$ is a relaxed $(r,r,m)$-SLMF. 
\end{proposition}

Since Theorem \ref{thm:main} gives a sufficient condition for $\Omega$ to be a base, and Proposition \ref{prp:necessary} gives a necessary condition, it is interesting to ask when these two coincide; for then, they would characterize the bases of the matroid. Note that only one direction needs to be checked: If $\Omega$ satisfies the partition of Theorem \ref{thm:main}, then it is a base set and so it is a relaxed $(r,r,m)$-SLMF by Proposition \ref{prp:necessary}. Lisa Nicklasson has recently communicated to me examples showing that these two conditions are not in general equivalent, not even after the assumption that every vertex of $G_\Omega$ has degree at least $r+1$ (Lemma \ref{lem:omega-r}). On the other hand, we have:

\begin{proposition} \label{prp:conj}
With $\#\Omega = r(m+n-r)$, the hypothesis of Theorem \ref{thm:main} is equivalent to the conclusion of Proposition \ref{prp:necessary} for i) $r=1$, ii) $r=m-1$,  iii) $r=m-2$; in these cases these conditions characterize the algebraic matroid of $\M(r,m \times n)$.
\end{proposition}

Interestingly, we also have:
\begin{proposition} \label{prp:Amini}
With $\#\Omega = r(m+n-r)$, the conditions of Theorem \ref{thm:main} and Proposition \ref{prp:necessary} are equivalent when $\# \omega_j = r+1, \,  \forall j \in [n]$; here necessarily $n=r(m-r)$.
\end{proposition}

Among all bases of the matroid of $\M(r,m \times n)$, it is of interest to characterize those for which the generic fiber consists of a single element. Our next result describes such a family of base sets. To state the result, note that if $\Omega_{\J}$ is a relaxed $(1,r,m)$-SLMF and denoting by $\Omega_j$ the set of subsets of $\omega_j$ of cardinality $r+1$, then there exist $\phi_{j'} \in \bigcup_{j \in \J} \Omega_j$ for $j' \in [m-r]$ such that $\Phi = \bigcup_{j' \in [m-r]} \phi_{j'} \times \{j'\}$ is an $(r,m)$-SLMF (Lemma \ref{lem:SLMF-induced}). We say that $\Omega_\J$ induces the $(r,m)$-SLMF $\Phi$. 

\begin{proposition} \label{prp:injective}
Suppose $\Omega$ satisfies the hypothesis of Theorem \ref{thm:main}, and that each $\Omega_{\J_\ell}$ induces the same $(r,m)$-SLMF $\Phi = \Phi_{\ell}$ for every $\ell \in [r]$. If $\operatorname{char}k =0$ or if $\operatorname{char}k=p>0$ is large enough, there exists a Zariski dense open set $U_\Omega \subset \M(r,m \times n)$, such that $\pi_\Omega^{-1}(\pi_\Omega(X)) = \{X\}$ for any $X \in U_\Omega$ ($X$ need not be closed or $k$-rational).  
\end{proposition}

Proposition \ref{prp:injective} is of significance to the problem of low-rank matrix completion \cite{kiraly2015algebraic}. There, a matrix $X$ of typically low rank $r$ is observed at only a subset $\Omega$ of its entries, and the problem is to correctly complete the matrix $\pi_\Omega(X)$. By verifying that $\Omega$ contains a subset $\Omega'$ that satisfies the conditions of Proposition \ref{prp:injective}, and assuming $X$ is sufficiently generic, one can ensure that any rank-$r$ completion of $\pi_\Omega(X)$ coincides with $X$ and thus is the correct completion, regardless of how it has been obtained. Inasmuch as understanding when $\pi_\Omega$ is generically finite is a more fundamental question than understanding when it is generically injective, the entire problem of characterizing the algebraic matroid of $\M(r, m \times n)$ is a fundamental question itself for the low-rank matrix completion problem. 

\subsection{Acknowledgments}
I am grateful to Aldo Conca for many inspiring discussions on the subject and beyond. I am also grateful to Omid Amini for communicating to me an initial proof of Proposition \ref{prp:Amini}; the elegant idea that appears in the proof here, that is to prove that the $(r,m)$-SLMF's are the base sets of a matroid and then to use the base-packing theorem of Edmonds \& Fulkerson, belongs to him. I am grateful to Lisa Nicklasson for sharing with me counterexamples to a conjecture in a previous version of this manuscript. I thank Daniel I. Bernstein and Louis Theran for useful exchanges on algebraic matroids and matrix completion, and I thank Olivier Wittenberg for a fruitful interaction regarding this paper. I acknowledge the support of the CAS Project for Young Scientists in Basic Research, Grant No. YSBR-034.

\subsection{Organization} 
In section \ref{section:SLMF} we review the connection between supports of linkage matching fields and local coordinates of the Grassmannian, which plays a crucial role in the paper. In section \ref{subsection:basic-facts} we prove ten basic lemmas that we will be using in the rest of the proofs. In particular, section \ref{subsection:Proof-main} proves Theorem \ref{thm:main}, and sections \ref{subsection:proof-prp:necessary}, \ref{subsection:proof-prp:conj}, \ref{subsection:proof-prp:Amini} and \ref{subsection:proof-prp:injective} prove Propositions \ref{prp:necessary}, \ref{prp:conj}, \ref{prp:Amini} and \ref{prp:injective}, respectively. Finally, section \ref{section:Examples} gives three examples illustrating section \ref{section:SLMF} and Theorem \ref{thm:main}.

%%%%
\section{Preliminaries: local coordinates on the Grassmannian induced by SLMF's} \label{section:SLMF}

We recall afresh the beautiful relationship between SLMF's and local coordinates on the Grassmannian, established by Sturmfels \& Zelevinsky in section 4 of \cite{sturmfels1993maximal}. Our exposition is from a certain dual point of view as needed for our purpose; for further details the reader is referred to \cite{sturmfels1993maximal}.

\subsection{Conventions} \label{subsection:Conventions} Let $T=k \bigl[\,[\psi]: \, \psi \subset [m], \, \# \psi=r\,\bigr]$ be the polynomial ring over $k$ generated by variables $[\psi]$'s, where $\psi$ ranges across all subsets of $[m]$ of size $r$. Let $\mathfrak{p}$ be the ideal of $T$ generated by the Pl{\"u}cker relations \cite{BrunsVetter:1988}. We set $\Gr(r,m) = \Proj(T/\mathfrak{p})$, the Grassmannian scheme of $r$-dimensional $k$-subspaces of $k^m$.

\begin{notation}
When we write $S \in \Gr(r,m)$, we will always mean $S$ to be a $k$-rational point of $\Gr(r,m)$, that is an $r$-dimensional $k$-subspace of $k^m$.
\end{notation}

\begin{notation}
When we say $B \in k^{m \times r}$ is a $k$-basis for $S \in \Gr(r,m)$, we will mean that the columns of $B$ are a $k$-basis of $S$.
\end{notation}

\begin{notation}
For $\mathcal{A}$ and $\mathcal{B}$ two subsets of $[m]$, we will denote by $\sigma(\mathcal{A},\mathcal{B})$ the number of elements of the form $(a,b) \in \mathcal{A} \times \mathcal{B}$ with $a>b$.
\end{notation}

Let $S \in \Gr(r,m)$ and $B \in k^{m \times r}$ a $k$-basis for $S$. With $\psi$ a subset of $[m]$ of size $r$, we shall write $[\psi]_{S}^B$ for the $r$-minor of $B$ corresponding to row indices $\psi$. We will further write $([\psi]_S: \, \psi \subset [m], \, \#\psi = r)$ for the set of Pl{\"u}cker coordinates of $S$ that identify $S$ with a point of $\mathbb{P}^{{m \choose r} -1}$; this is the class in $\mathbb{P}^{{m \choose r} -1}$ of $([\psi]_S^B: \, \psi \subset [m], \, \#\psi = r)$ for any $k$-basis $B$ of $S$. Now let $T'=k \bigl[\,[\chi]: \, \chi \subset [m], \, \# \chi=m-r\,\bigr]$ be the polynomial ring over $k$ generated by variables $[\chi]$'s, where $\chi$ ranges across all subsets of $[m]$ of size $m-r$. The $k$-algebra isomorphism $T'\rightarrow T$ that takes $[\chi]$ to $\sigma([m] \setminus \chi,\chi)  \,  \big[[m] \setminus \chi \big]$, maps $\mathfrak{p}$ to the ideal $\mathfrak{p}'$ of Pl\"ucker relations in $T'$. It thus induces a canonical isomorphism of projective schemes $\Gr(r,m) \rightarrow \Gr(m-r,m)$. By working locally on the standard affine open sets of the Grassmannian, it can be seen that, on the level of $k$-rational points, $S \in \Gr(r,m)$ is mapped to its orthogonal complement $S^\perp \in \Gr(m-r,m)$; the latter being the set of all vectors $y=(y_i)_{i \in [m]}$ in $k^m$, such that for every vector $x=(x_i)_{i \in [m]} \in S$ we have $\sum_{i\in [m]} x_iy_i=0$. Hence, the Pl\"ucker coordinates of $S^\perp$ can be expressed in terms of the Pl\"ucker coordinates of $S$ as $[\chi]_{S^\perp} = \sigma([m] \setminus \chi,\chi) \big[[m] \setminus \chi \big]_S$.

%%%
\subsection{SLMF's and maximal minors} The combinatorial object of an $(r,m)$-SLMF has the following important algebraic characterization:

\begin{proposition} \label{prp:SLMF-algebraic}
Let $\Phi = \bigcup_{j \in [m-r]} \phi_j \times \{j\}$ be a subset of $[m] \times [m-r]$ with $\# \phi_j = r+1$ for every $j \in [m-r]$. We have that $\Phi$ is an $(r,m)$-SLMF if and only if all maximal minors of a generic matrix with support on $\Phi$ over $k$ are non-zero.
\end{proposition}

%%%%
\subsection{SLMF's and local coordinates on $\Gr(r,m)$} \label{subsection:SLMFs} In the following discussion we assume that $r<m-1$; see also Remark \ref{rem:hyperplane}. Let $S \in \Gr(r,m)$ and let the matrix $A \in k^{m \times (m-r)}$ contain a basis of $S^\perp$ in its columns, expressed with respect to the standard basis of $k^m$. Let $\Phi = \bigcup_{j \in [m-r]} \phi_j \times \{j\}$ be an $(r,m)$-SLMF. For $j \in [m-r]$ denote by $H_j$ the $k$-subspace of $k^{m-r}$ spanned by the $m-r-1$ rows of $A$ indexed by $[m]\setminus \phi_j$. Consider the locus $V_\Phi \subset \Gr(r,m)$ consisting of those points $S \in \Gr(r,m)$ with the property that all $H_j$'s are codimension-$1$ hyperplanes of $k^{m-r}$ intersecting only at the origin, that is $\bigcap_{j \in [m-r]} H_j=0$; $V_\Phi$ is open, because it is open on each set of the standard affine open cover of $\Gr(r,m)$. 

For any $j, j' \in [m-r]$, let $\mathfrak{m}_{j j'}$ be the $(m-r-1)$-minor of $A$ corresponding to row indices $[m] \setminus \phi_j$ and column indices $[m-r] \setminus \{j'\}$; define a matrix $M \in k^{(m-r) \times (m-r)}$ as $M=(\mathfrak{m}_{j j'}: \, j,  j' \in [m-r])$. For any $j \in [m-r]$, $H_j$ has codimension $1$ if and only if not all $\mathfrak{m_{jj'}}$'s are zero as $j'$ varies in $[m-r]$. In that case, the vector $[\mathfrak{m}_{j1}, \,  -\mathfrak{m}_{j2}, \,   \mathfrak{m}_{j3}, \,  \cdots,\,  (-1)^{m-r-1}\mathfrak{m}_{j(m-r)}]$ is normal to the hyperplane $H_j$. Hence, a point $x=[x_1,\dots,x_{m-r}] \in k^{m-r}$ lies in $\bigcap_{j \in [m-r]} H_j$ if and only if the point $[x_1,-x_2,\dots,(-1)^{m-r-1}x_{m-r}]$ is orthogonal to all rows of $M$. We conclude that $S \in V_\Phi$ if and only if $\det(M) \neq 0$. 

Suppose $S \in V_\Phi$. Then the normal vectors to the $H_j$'s mentioned above form a $k$-basis of $k^{m-r}$; hence there is an automorphism of $k^{m-r}$ that takes $H_j$ to the hyperplane with normal vector the $j$th standard vector $e_j$ of $k^{m-r}$. Applying this change of basis to the matrix $A$, we see that $S^\perp$ can also be represented by some $B \in k^{m \times (m-r)}$, whose rows indexed by $[m]\setminus \phi_j$ are orthogonal to $e_j$; that is $B=(b_{ij})$ has support on $\Phi$. For every $j, j' \in [m-r]$ let $\mathfrak{n}_{j j'}$ be the $(m-r-1)$-minor of $B$ corresponding to row indices in $[m] \setminus \phi_j$ and column indices $[m-r]\setminus \{j'\}$. Since the $j$th column of $B$ has zeros at locations $[m]\setminus \phi_j$, we see that $\mathfrak{n}_{j j'}=0$ for every $j' \in [m-r]$ except possibly for $j'=j$; thus the matrix $N=(\mathfrak{n}_{j j'}: \, j,  j' \in [m-r])$ is diagonal. Moreover, $N=MC$, where $C$ is an invertible matrix; hence all $\mathfrak{n}_{jj}$'s are non-zero because $\det(M) \neq 0$. With $\phi_{ij}$ the $i$th element of $\phi_j$, we consider the $(m-r) \times (m-r)$ submatrix of $B$ corresponding to row indices $\chi=[m]\setminus ( \phi_j \setminus \{\phi_{ij}\})$; column $j$ of this submatrix has zeros everywhere, except possibly at the entry $b_{\phi_{ij}\, j}$. Thus the determinant of this submatrix is 
\begin{align*}
[\chi]_{S^\perp}^B = \sigma(\{\phi_{ij}\},\chi)(-1)^{j-1} b_{\phi_{ij} j} \mathfrak{n}_{jj}.
\end{align*}
\noindent Passing from the Pl{\"u}cker coordinates of $S^\perp$ to the Pl{\"u}cker coordinates of $S$, and since $B$ still stays a basis of $S^\perp$ after scaling each column by any non-zero element of $k$, we have that when $S \in V_\Phi$ then $S^\perp$ has a $k$-basis $B$ with $j$th column given by 
\begin{align}
b_{\phi_{ij} j} = (-1)^{i-1}[ \phi_j \setminus \{\phi_{ij}\}]_{S}^E  \, \, \, \text{for} \, \, \,  i \in [r+1] \, \, \, \, \, \, \text{and} \, \, \, \, \, \, b_{\ell j } = 0 \, \, \, \text{for} \, \, \,  \ell \not\in \phi_j, &\label{eq:Atilde}
\end{align} where $E \in k^{m \times r}$ is any basis of $S$.

Conversely, suppose that the matrix $B$ in \eqref{eq:Atilde} has rank $m-r$. Since the $j$th column of $B$ is orthogonal to $S$, we see that $B$ is a basis for $S^\perp$. Let us set $A=B$. Not all elements of the $j$th column of $A$ are zero, so say $[\phi_j \setminus \{\phi_{ij}\}]_{S}^E \neq 0$ for some $i$. In the notation above, this is up to sign equal to $a_{\phi_{ij}j}$ times $\mathfrak{m}_{jj}$; that is $\mathfrak{m}_{jj} \neq 0$. Also, the matrix $M$ is diagonal, from which we see that $S \in V_\Phi$. We thus have:

\begin{proposition} \label{prp:basis-Sperp}
$S \in V_\Phi$ if and only if the matrix $B$ of \eqref{eq:Atilde} has full rank. In this case $B$ is a basis for $S^\perp$.
\end{proposition}

\noindent Note that $V_\Phi$ is non-empty by Proposition \ref{prp:SLMF-algebraic}: if $B$ is a generic matrix over $k$ with support in $\Phi$, all its maximal minors are non-zero and so are all $\mathfrak{n}_{jj}$'s; it follows that the orthogonal complement of the column-space of $B$ lies in $V_\Phi$.

Next, consider the matrix 
$$\Pi_\Phi = \left((-1)^{\sigma(\{\alpha\},\phi_\beta)} [\phi_\beta \setminus  \{ \alpha \} ]:  \, \alpha \in [m] \setminus \phi_1, \, \beta \in [m-r]\setminus \{1\} \right) \in T^{(m-r-1)\times (m-r-1)},$$ 

\noindent with $[\phi_\beta \setminus  \{ \alpha \} ]$ set to zero whenever $\alpha \notin \phi_\beta$. Denote also $\Pi_{\Phi,S}^E$ the matrix $\Pi_\Phi$ after the substitution of each variable $[\psi] \in T$ by $[\psi]_S^E$. The determinant of $\Pi_{\Phi,S}^E$ is the $(m-r-1)$-minor of the matrix $B$ of \eqref{eq:Atilde} corresponding to row indices $[m] \setminus \phi_1$ and column indices $[m-r]\setminus \{1\}$. Suppose that $\det(\Pi_{\Phi,S}^E) \neq 0$; then columns $2,\dots,m-r$ of $B$ are part of a basis of $S^\perp$ that can be completed to a full basis of $S^\perp$, by adjoining a suitable vector which can be taken to have support on $\phi_1$. Necessarily, this vector is equal up to scale to the first column of $B$ in \eqref{eq:Atilde}, whence $B$ has full rank; we have:

\begin{proposition} \label{prp:p-Phi}
Let $\Phi$ be an $(r,m)$-SLMF. Then $S \in V_\Phi$ if and only if $S$ lies away from the hypersurface of $\Gr(r,m)$ defined by the polynomial $p_\Phi = \det(\Pi_\Phi).$
\end{proposition}

Finally, define a morphism $\gamma_{\phi_j}: V_\Phi  \rightarrow \mathbb {P}^{m-1}$ by taking $S \in V_\Phi$ with basis $B$ as in \eqref{eq:Atilde}, to the class in $\mathbb {P}^{m-1}$ of the $j$th column of $B$; define $\gamma_\Phi: V_\Phi \rightarrow \prod_{j \in [m-r]} \mathbb{P}^{m-1}$ by taking $S \in V_\Phi$ to $(\gamma_{\phi_j}(S): \, j \in [m-r])$. In view of Proposition \ref{prp:basis-Sperp}, the following is now clear:

\begin{proposition} \label{prp:SZ93}
With $\Phi$ an $(r,m)$-SLMF, $\gamma_\Phi$ is an open embedding, mapping an $r$-dimensional $k$-subspace $S \in V_\Phi$ to $m-r$ linearly independent directions in $S^\perp$.   
\end{proposition} 

\begin{remark} \label{rem:hyperplane}
When $r=m-1$, the open set $V_\Phi$ can be taken to be the entire space $\Gr(r,m)$ and the statement of Proposition \ref{prp:basis-Sperp} becomes the well-known formula for the normal vector to a codimension-$1$ hyperplane; the statement of Proposition \ref{prp:SZ93} becomes the standard assertion $\Gr(m-1,m) \cong \mathbb{P}^{m-1}$. 
\end{remark}

Example \ref{ex:SZ93} in section \ref{section:Examples} illustrates some of the above concepts.

%%%%%%%%%%%%
\section{Proofs} \label{section:Proofs} 

In this section we provide the proofs to the main results of this paper stated in section \ref{section:Introduction}. Throughout, $\Omega$ is always a subset of $[m]\times [n]$ with the convention of Notation \ref{not:Omega}. For $j \in [n]$ let $\Omega_j$ be the set of all subsets of $\omega_j$ of cardinality $r+1$. We will also be using the notation of section \ref{section:SLMF}. 

%%%%
\subsection{Basic facts} \label{subsection:basic-facts}

We begin with some preparations. For $\omega \subseteq [m]$ define the coordinate projection $\pi_\omega: k^m \rightarrow k^{\#\omega}$ by sending $(\xi_i)_{i \in [m]}$  to $(\xi_i)_{i \in \omega}$. For $B \in k^{m \times r}$ let $\pi_\omega(B) \in k^{\# \omega \times r}$ be the matrix obtained by applying $\pi_\omega$ to the columns of $B$.  

\begin{lemma} \label{lem:x-reconstruction}
Let $S \in \Gr(r,m)$, let $\omega \subseteq [m]$ with $\# \omega \ge r$, and suppose that $\dim_k \pi_{\omega}(S) = r$. Suppose that for some $x \in k^m$ we have $\pi_\omega(x) \in \pi_\omega(S)$. Then there exists unique $y \in S$ such that $\pi_\omega(y) = \pi_\omega(x)$.
\end{lemma}
\begin{proof}
Let $B \in k^{m \times r}$ be a basis for $S$. By hypothesis $\pi_\omega(B) \in k^{\#\omega \times r}$ is a basis of $\pi_\omega(S)$. Then there is a unique $c \in k^r$ such that $\pi_\omega(x) = \pi_\omega(B) c$. Define $y=Bc$, clearly $\pi_\omega(x) = \pi_\omega(y)$. Suppose $\pi_\omega(y) = \pi_\omega(y')$ for some other $y' \in S$. There is a unique $c' \in k^r$ such that $y' = Bc'$. On the other hand, the equation $\pi_\omega(y) = \pi_\omega(y')$ implies that $\pi_\omega(B) (c-c') = 0$. But $\pi_\omega(B)$ has rank $r$ and so $c = c'$. 
\end{proof}

\begin{lemma} \label{lem:x-reconstruction-union}
Let $S \in \Gr(r,m)$, $x \in k^m$ and $\omega, \omega' \subseteq [m]$ with $\pi_\omega(x) \in \pi_\omega(S)$ and $\pi_{\omega'}(x) \in \pi_{\omega'}(S)$. If $\dim_k \pi_{\omega \cap \omega'}(S) = r$ then $\pi_{\omega \cup \omega'}(x) \in \pi_{\omega \cup \omega'}(S)$.
\end{lemma}
\begin{proof}
There exist $y, y' \in S$ such that $\pi_\omega(x) = \pi_\omega(y)$ and $\pi_{\omega'}(x) = \pi_{\omega'}(y')$. This implies that $\pi_{\omega \cap \omega'}(x) = \pi_{\omega \cap \omega'}(y) = \pi_{\omega \cap \omega'}(y')$. Lemma \ref{lem:x-reconstruction} gives $y = y'$. Now $y$ agrees with $x$ at each coordinate $i \in \omega \cup \omega'$, that is $\pi_{\omega \cup \omega'}(x) = \pi_{\omega \cup \omega'}(y) \in \pi_{\omega \cup \omega'}(S)$.
\end{proof}

\begin{lemma} \label{lem:sections}
Let $\phi = \{i_1<\dots<i_{r+1}\} \subseteq [m]$, let $x \in k^m$ and $S \in \Gr(r,m)$ with $\dim_k \pi_{\phi}(S)=r$. Then $\pi_\phi(x) \in \pi_\phi(S)$ if and only if $\sum_{\alpha \in [r+1]} \, (-1)^{\alpha+1} \, x_{i_\alpha} \, [\phi \setminus \{i_\alpha\}]_{S}=0$.
\end{lemma}
\begin{proof}
With $B \in k^{m \times r}$ a basis for $S$, we use $[\phi \setminus \{i_\alpha\}]_S^B$ to identify $\det \big(\pi_{\phi \setminus \{i_\alpha\}}(B) \big)$ for every $\alpha \in [r+1]$ (see section \ref{subsection:Conventions}). Applying Laplace expansion on the first column of the matrix $[\pi_\phi(x) \, \, \pi_\phi(B)] \in k^{(r+1) \times (r+1)}$, shows that $\det \big([\pi_\phi(x) \, \, \pi_\phi(B)]\big)=0$ is equivalent to the formula in the statement. Since $\pi_{\phi}(B)$ has rank $r$ by hypothesis, $\det \big([\pi_\phi(x) \, \, \pi_\phi(B)]\big)=0$ is equivalent to $\pi_\phi(x) \in \pi_\phi(S)$.
\end{proof}

\begin{lemma} \label{lem:NonDimensionDrop}
With $\Phi=\bigcup_{j \in [m-r]} \phi_j \times \{j\}$ an $(r,m)$-SLMF and $S \in V_\Phi$, we have that $\dim_k \pi_{\phi_j}(S) = r$ for every $j \in [m-r]$.
\end{lemma} 
\begin{proof}
Proposition \ref{prp:basis-Sperp} asserts that the matrix $B$ of \eqref{eq:Atilde} is a basis of $S^\perp$. On the other hand, $\dim_k \pi_{\phi_j}(S)<r$ if and only if all Pl{\"u}cker coordinates $[\phi_j \setminus \{\phi_{i j}\}]_S$ are zero, where $\phi_{ij}$ denotes the $i$th element of $\phi_j$. But in that case the $j$th column of $B$ would be zero, a contradiction.
\end{proof}

\begin{lemma} \label{lem:PointProjection}
With $\Phi=\bigcup_{j \in [m-r]} \phi_j \times \{j\}$ an $(r,m)$-SLMF and $S \in V_\Phi$, if $\pi_{\phi_j}(x) \in \pi_{\phi_j}(S)$ for every $j \in [m-r]$, then $x \in S$. 
\end{lemma}
\begin{proof}
By Lemma \ref{lem:NonDimensionDrop}, $\dim_k \pi_{\phi_j}(S) = r$ for every $j \in [m-r]$. Then Lemma \ref{lem:sections} implies that the relation $\pi_{\phi_j}(x) \in \pi_{\phi_j}(S)$ is equivalent to $x$ being orthogonal to the $j$th column of $B$, where $B$ is the matrix given in \eqref{eq:Atilde}. Since this is true for every $j \in [m-r]$ and since by Proposition \ref{prp:basis-Sperp} the columns of $B$ form a basis for $S^\perp$, certainly $x \in S$. 
\end{proof}

\begin{lemma} \label{lem:k->K}
Let $k \hookrightarrow K$ be a field extension. Then the algebraic matroid of $k[Z]/I_{r+1}(Z)$ coincides with the algebraic matroid of $k[Z]/I_{r+1}(Z) \otimes_k K$. 
\end{lemma}
\begin{proof}
An $\Omega \subset [m] \times [n]$ is an independent set in the algebraic matroid of $M(r, m\times n)$ if and only if the ring homomorphism $k[Z_\Omega] \rightarrow k[Z]/I_{r+1}(Z)$ is injective. Since $K$ is a faithfully flat $k$-module, this is equivalent to the injectivity of $K[Z_\Omega] \rightarrow k[Z]/I_{r+1}(Z) \otimes_k K$.
\end{proof}

\begin{lemma} \label{lem:omega-r}
Suppose that $\# \omega_t = r$ for some $t \in [n]$, and set $\Omega' = \bigcup_{j \in [n-1]} \omega_j' \times \{j\} \subset [m] \times [n-1]$, where $\omega_j'=\omega_j$ for $j<t$ and $\omega_j' = \omega_{j+1}$ for $t \le j <n$. Then $\Omega$ is a base of the algebraic matroid of $\M(r,m \times n)$ if and only if $\Omega'$ is a base of the algebraic matroid of $\M(r,m \times (n-1))$.
\end{lemma}
\begin{proof}

For any invertible matrices $P_1 \in k^{m\times m}$ and $P_2 \in k^{n \times n}$ we have that 
$ I_{r+1}(Z) =  I_{r+1}(P_1 Z P_2)$ \cite{BrunsVetter:1988}. Taking $P_1,P_2$ to be permutations, we may assume that $t=1$ and $\omega_1 = \{m-r+1,\dots,m\}$. 

We write $Z = [z \, \, \, Z']$ where $z$ is the first column of $Z$. Recall that $Z_\Omega=\{z_{ij}:  \, (i,j) \in \Omega\}$. Set $z_{\omega_1} = \{z_{(m-r+1)1},\dots,z_{m1} \}$ and $Z'_\Omega=\{z_{ij}:  \, (i,j) \in \Omega, \, j>1\}$. We denote by $k[z]$, $k[Z']$, $k[z_{\omega_1},Z'], \, k[Z_\Omega]$ and $k[Z'_{\Omega}]$ the polynomial rings over $k$ generated by the corresponding variables.

We consider the lexicographic term order on $k[Z]$ with $$z_{11}> z_{21} > \cdots > z_{m1} > z_{12}> z_{22}> \cdots > z_{m2} > z_{13}>z_{23}> \cdots >z_{m-1,n}> z_{mn}.$$ 

\noindent This is a so-called diagonal term order, in the sense that the leading term of every minor of the matrix $Z$ is the product of the variables on the main diagonal of the minor. With respect to this order $>$, it is well-known that the $(r+1)$-minors of $Z$ form a Gr{\"o}bner basis of the ideal $ I_{r+1}(Z)$ \cite{bruns2022determinants}. By elimination theory (e.g., Section 15.10.4 in \cite{Eisenbud-2004} or Chapter 3 in \cite{cox2015ideals}), a Gr{\"o}bner basis of the ideal $k[z_{\omega_1}, Z'] \cap  I_{r+1}(Z)$ with respect to the restriction of $>$ on $k[z_{\omega_1}, Z']$ is given by the elements of the above Gr{\"o}bner basis of $ I_{r+1}(Z)$ that live in $k[z_{\omega_1}, Z']$. These are the $(r+1)$-minors of $Z'$, and so we have the inclusion of sets $$k[Z_\Omega] \cap  I_{r+1}(Z) \subset k[z_{\omega_1}, Z'] \cap  I_{r+1}(Z) \subset  I_{r+1}(Z').$$ 

Now, take a polynomial $p \in k[Z_\Omega] \cap  I_{r+1}(Z)$. By the above inclusion $p \in   I_{r+1}(Z')$, that is the class $\overline{p}$ of $p$ in $k[Z] /  I_{r+1}(Z')$ is zero. Denote by $ I'_{r+1}(Z')$ the ideal of $k[Z']$ generated by the $r+1$ minors of $Z'$. Then $k[Z] /  I_{r+1}(Z') \cong  k[Z'] /  I'_{r+1}(Z') \otimes_k k[z]$. Thus $k[Z] /  I_{r+1}(Z')$ is a free $k[Z'] /  I'_{r+1}(Z')$-module with basis all the monomials $z^\alpha = z_{11}^{\alpha_1}\cdots z_{m1}^{\alpha_m}$ in the variables $z$, where $\alpha = (\alpha_1,\dots,\alpha_m)$ is a multi-exponent. Let us write $p = \sum_{\alpha} c_\alpha \, z^\alpha$, where the $c_\alpha$'s are polynomials in $k[Z']$. Let us also write $\overline{p} = \sum_{\alpha} \overline{c}_\alpha \, z^\alpha$, where $\overline{c}_\alpha$ is the class of $c_{\alpha}$ in $k[Z'] /  I'_{r+1}(Z')$. Since $\overline{p}=0$ in $k[Z] /  I_{r+1}(Z')$, the freeness of $k[Z] /  I_{r+1}(Z')$ over $k[Z'] /  I'_{r+1}(Z')$ gives that $\overline{c}_\alpha = 0$ for every $\alpha$ that appears in the summation. Equivalently, $c_\alpha \in  I'_{r+1}(Z')$ for every $\alpha$. Since $p \in k[Z_\Omega]$ we have that $c_\alpha \in k[Z_\Omega'] \cap  I'_{r+1}(Z')$. Now, $p$ is non-zero if and only if not all $c_\alpha$'s are zero. That is, $k[Z_\Omega] \cap  I_{r+1}(Z)$ is non-zero if and only if $k[Z_\Omega'] \cap  I'_{r+1}(Z')$ is non-zero. Equivalently, the variables $Z_\Omega$ are algebraically independent mod $ I_{r+1}(Z)$ if and only if the variables $Z'_\Omega$ are algebraically independent mod $ I'_{r+1}(Z')$, which is the statement of the lemma. 
\end{proof}

\begin{lemma} \label{lem:SLMF-I}
Let $\Phi =\bigcup_{j \in [m-r]} \phi_j \times \{j\}$ be a subset of $[m] \times [m-r]$ with $\# \phi_j=r+1$ for every $j \in [m-r]$. Then $\Phi$ is an $(r,m)$-SLMF if and only if it is a relaxed $(1,r,m)$-SLMF.
\end{lemma}
\begin{proof}
Suppose first that $\Phi$ is an $(r,m)$-SLMF. Let $\mathcal I \subset [m]$ with $\# \mathcal I \ge r+1$. Define $\mathcal J = \{ j \in [m-r]: \, \phi_j \subset \mathcal I\}$. If $\mathcal J$ is empty, we have $\max\big\{\# (\phi_j \cap \mathcal I) -r, 0 \big\}=0$ for every $\phi_j$ and thus the $(1,r,m)$-SLMF condition is satisfied for this $\mathcal I$. If $\mathcal J$ is non-empty, by the defining property \eqref{eq:SLMF} of $(r,m)$-SLMF, we must have $\# \bigcup_{j \in \mathcal J} \phi_j \ge \# \mathcal J + r$. But $\bigcup_{j \in \mathcal J} \phi_j \subset \mathcal I$ so that $\# \mathcal J  \le \# \mathcal I - r$. On the other hand, $\# \mathcal J $ is equal to the number of $\phi_j$'s contained in $\mathcal I$, which is equal to $\sum_{j \in [m-r]} \max\big\{\# (\phi_j \cap \mathcal I) -r, 0 \big\}$. This proves the inequality in Definition \ref{dfn:relaxedSLMF} for any $\mathcal I$. The equality for $\mathcal I = [m]$ is evident, since this is saying that the total number of $\phi_j$'s is $m-r$.

Conversely, suppose the inequality in Definition \ref{dfn:relaxedSLMF} is true with $\nu=1$ for any $\mathcal I \subset [m]$. For any $\mathcal J \subset [n]$, set $\mathcal I = \bigcup_{j \in \mathcal J} \phi_j$. Then $\# \mathcal J \le \sum_{j \in [m-r]} \max\big\{\# (\phi_j \cap \mathcal I) -r, 0 \big\} \le \# \bigcup_{j \in \mathcal J} \phi_j -r$, that is, the $(r,m)$-SLMF condition \eqref{eq:SLMF} is satisfied.
\end{proof}

\begin{lemma} \label{lem:SLMF-induced}
Suppose $\Omega_\J = \bigcup_{j \in \J} \omega_j \times \{j\}$ is a relaxed $(1,r,m)$-SLMF for some $\J \subset [n]$. Denote by $\Omega_j$ the set of subsets of $\omega_j$ of cardinality $r+1$. Then there exist $\phi_{j'} \in \bigcup_{j \in \J} \Omega_j$ for $j' \in [m-r]$ such that $\Phi = \bigcup_{j' \in [m-r]} \phi_{j'} \times \{j' \}$ is an $(r,m)$-SLMF.
\end{lemma}
\begin{proof}
Define $\mathcal J' = \{j \in \mathcal J: \, \# \omega_j > r \}$. For every $j \in \mathcal J'$ fix any $\psi_j \subset \omega_j$ with $\# \psi_j = r$, and for every $t \in \omega_j \setminus \psi_j$ define $\phi_{j, t}  = \psi_j \cup \{t\}$. For $j \in \mathcal J \setminus \mathcal J'$ set $\psi_j  = \emptyset$. Setting $\I = [m]$ in the Definition \ref{dfn:relaxedSLMF} of relaxed $(1,r,m)$-SLMF gives $\sum_{j \in \J} \max\{\# (\omega_j \setminus \psi_j) -r, 0\} = m-r$. Thus in total we have $m-r$ $\phi_{j,t}$'s and so we can order them as $\phi_1,\dots,\phi_{m-r}$. Set $\Phi=\bigcup_{j \in [m-r]} \phi_{j} \times \{j\}$. Now, let $\mathcal I $ be any subset of $[m]$ with $\# \mathcal I \ge r+1$. Note that for any $j \in \mathcal J'$ we have that $\max\big\{\# (\phi_{j,t} \cap \mathcal I) -r, 0 \big\} $ is zero whenever $\phi_{j,t}$ is not contained in $\mathcal I$, and it is equal to $1$ otherwise. On the other hand, for $j \in \mathcal J'$ the quantity $\# (\omega_j \cap \mathcal I ) -r$ is at least equal to the number of $\phi_{j,t}$'s that are contained in $\mathcal I$. Hence, $\# \mathcal I -r \ge \sum_{j \in \mathcal J} \max\big\{\# (\omega_j \cap \mathcal I) -r, 0 \big\}  = \sum_{j \in  \mathcal J'} \max\big\{\# (\omega_j \cap \mathcal I) -r, 0 \big\} \ge \sum_{j \in [m-r]} \max\big\{\# (\phi_j \cap \mathcal I) -r, 0 \big\}$. Thus $\Phi$ is a relaxed $(1,r,m)$-SLMF and so by Lemma \ref{lem:SLMF-I}  $\Phi$ is an $(r,m)$-SLMF.
\end{proof}

\begin{lemma} \label{lem:less-than-r}
Suppose $\# \Omega = r(m+n-r)$ but there is some $j \in [n]$ such that $\#\omega_j < r$. Then $\Omega$ is not a base of the algebraic matroid of $\M(r, m \times n)$. 
\end{lemma}
\begin{proof}
Let $X \in \M(r, m \times n)$ be any rank-$r$ $k$-rational point. Let $B \in k^{m \times r}$ be a basis of the column-space of $X$. With $x_j$ the $j$th column of $X$, the $\#\omega_j \times r$ linear system of equations $\pi_{\omega_j}(x_j) = \pi_{\omega_j}(B) c$ in the unknown $c$ is underdetermined, so it has infinitely many solutions. For any such solution $c$, let $X_c$ be the $m \times n$ matrix whose $j$th column is $Bc$ and otherwise it coincides with $X$. Since $Bc$ is in the column-space of $X$, the matrix $X_c$ has rank $r$ and by construction it agrees with $X$ on $\Omega$, that is $X_c$ is in the fiber of $\pi_\Omega$ over $X$. We conclude that $\pi_\Omega^{-1}(\pi_\Omega(X))$ is an infinite set for every $k$-rational point $X$ of rank $r$. 
Now $\M(r, m \times n)$ is a rational variety; for instance the localization $(k[Z]/I_{r+1}(Z))_p$ where $p$ is the determinant of the top leftmost $r \times r$ submatrix of $Z$ is the localization of the polynomial ring $k[z_{ij}: \, \min\{i,j\}\le r]$ at $p$. As such, the set of $k$-rational points is dense and so it will intersect the non-empty open locus where the fiber of $\pi_\Omega$ achieves its minimal dimension. That is, the generic fiber will have positive dimension and by the discussion in section \ref{subsection:TheAlgebraicMatroid}, $\Omega$ can not be a base set.
\end{proof}

%%%%
\subsection{Proof of Theorem \ref{thm:main}} \label{subsection:Proof-main}

Let us deal with the codimension-$1$ case first, where $r = m-1$ and $I_{r+1}(Z)$ is the ideal of maximal minors of $Z$. Recall from section \ref{subsection:known-facts} that the maximal minors are a universal Gr\"obner basis for $I_m(Z)$, which implies that an $\Omega$ with $\#\Omega = (m-1)(n+1)$ is a base set of the algebraic matroid of $\M(m-1, m \times n)$ if and only if $\mathcal{G}_\Omega$ does not contain the complete bipartite graph $K_{m,m}$. Now, $\Omega_{\mathcal{J}_\ell}$ is a relaxed $(1,m-1,m)$-SLMF if and only if there is exactly one $j \in \mathcal{J}_\ell$ such that $\omega_j = [m]$. Thus the hypothesis implies that there are exactly $m-1$ $\omega_j$'s in $\Omega$ with cardinality $m$. Hence $\mathcal{G}_\Omega$ can not contain $K_{m,m}$, and this settles the codimension-$1$ case.

Now we claim that for any $r$ as in Notation \ref{not:rmn} the hypothesis implies $\#\omega_j \ge r$ for every $j \in [n]$. To see this, note that by the definition of a relaxed $(1,r,m)$-SLMF $\Omega_{\mathcal{J}_\ell}$ we have $\sum_{j \in \J_\ell} \max\{\#\omega_j-r,0\} = m-r$, and so $\sum_{\ell \in [r]} \sum_{j \in \J_\ell} \max\{\#\omega_j-r,0\}=r(m-r)$. On the other hand, $\#\Omega = r(m+n-r) = \sum_{j \in [n]} \#\omega_j$ and so $\sum_{j \in [n]} (\#\omega_j -r) = r(m-r)$. Necessarily $\#\omega_j -r = \max\{\#\omega_j-r,0\}$ for every $j \in [n]$, i.e. $\#\omega_j \ge r, \, \forall j \in [n]$. Thus in view of Lemma \ref{lem:omega-r}, we may assume that $\# \omega_j \ge r+1$ for every $j \in [n]$; however after this reduction it may no longer be true that $m \le n$. 

For the rest of the proof we assume $r \le m-2$. In view of Lemma \ref{lem:k->K} we may assume that $k$ is algebraically closed. By Lemma \ref{lem:SLMF-induced}, for every $\ell \in [r]$ there are $\phi_{j}^\ell \in \bigcup_{t \in \J_\ell} \Omega_{t}, \, j \in [m-r]$, such that $\Phi_\ell = \bigcup_{j \in [m-r]} \phi_{j}^\ell \times \{j\}$ is an $(r,m)$-SLMF. Recalling the notation of section \ref{section:SLMF}, for a closed point $X \in \M(r,m \times n)$ and $S$ the column-space of $X$, the condition $S \in \bigcap_{\ell \in [r]} V_{\Phi_\ell}$ is true on a dense open set of $\M(r,m \times n)$, which can be described as follows. Let $p =\prod_{\ell \in [r]} p_{\Phi_\ell}$ where $p_{\Phi_\ell} \in T$ is the polynomial in Pl{\"u}cker coordinates given in Proposition \ref{prp:p-Phi}. For any $\psi \subseteq [n]$ with $\# \psi = r$, replace every $[\phi_\beta^\ell \setminus \{\alpha\}]$ in $p$ by the $r \times r$ minor of $Z$ with row indices $\phi_\beta^\ell \setminus \{\alpha\}$ and column indices $\psi$, to obtain a polynomial $p_\psi \in k[Z]$. Varying $\psi$ gives the open set $U_{\Omega} = \bigcup_{\psi \subseteq [n], \#\psi = r} \Spec(k[Z]/I_{r+1}(Z))_{\overline{p}_\psi}$ of $\M(r,m \times n)$, where $(k[Z]/I_{r+1}(Z))_{\overline{p}_\psi}$ is the localization of $k[Z]/I_{r+1}(Z)$ at the multiplicatively closed set $\{1,\overline{p}_\psi,\overline{p}_\psi^2,\dots\}$, with $\overline{p}_\psi$ the class of $p_\psi$ in $k[Z]/I_{r+1}(Z)$. Then $S \in \bigcap_{\ell \in [r]} V_{\Phi_\ell}$ if and only if $X \in  U_{\Omega}$. To see that $U_\Omega$ is non-empty, first note that $\bigcap_{\ell \in [r]} V_{\Phi_\ell}$ is non-empty, because it is the intersection of finitely many dense open sets of the irreducible projective scheme $\Gr(r,m)$. Let $S \in \bigcap_{\ell \in [r]} V_{\Phi_\ell}$ be any closed point and $s_\ell \in k^m, \, \ell \in [r]$, a $k$-basis for $S$. Define $X \in \M(r,m \times n)$ by setting $x_j = s_\ell$ whenever $j \in \J_\ell$.  Now choose a $\psi \subset [n]$ that contains exactly one index from each $\J_\ell$. Then $p_\psi(X)$ is the evaluation of the polynomial $p_\psi$ at the Pl\"ucker coordinates of the subspace $S$ computed from its $k$-basis of $s_\ell$'s; by the choice of $S$ we have that $p_\psi(X)\neq 0$ and so $X \in U_\Omega$.  

Let $\pi_\Omega': U_\Omega \rightarrow \mathbb{A}^{\Omega}$ be the restriction of $\pi_\Omega$ to $U_\Omega$. Let $X'=[x_1' \cdots x_n']  \in \pi_\Omega'^{-1} \big( \pi'_\Omega(X) \big)$ be a closed point in the fiber over the $X=[x_1 \cdots x_n]$ defined above; here $x_j$ and $x_j'$ is the $j$th column of $X$ and $X'$, respectively. Let $S'$ be the column-space of $X'$. Since $x_j' \in S'$ we also have $\pi_{\omega_j}(x_j') \in \pi_{\omega_j}(S')$ for any $ j \in [n]$. Moreover, $\pi_{\omega_j}(x_j') = \pi_{\omega_j}(x_j)$ for every $j \in [n]$. By the definition of $X$ we conclude that $\pi_{\phi_j^\ell}(s_\ell) \in \pi_{\phi_j^\ell}(S')$ for every $j \in [m-r]$ and every $\ell \in [r]$. Since $X' \in U_\Omega$, we have $S' \in \bigcap_{\ell \in [r]} V_{\Phi_\ell}$. Then by Lemma \ref{lem:PointProjection} we must have that $s_\ell \in S'$ for every $\ell \in [r]$. But then $S' = S$. Now take any $j \in [n]$ and recall the assumption $\#\omega_j \ge r+1$. Suppose that $j \in \mathcal J_\ell$ for some $\ell \in [r]$. By the proof of Lemma \ref{lem:SLMF-induced}, there is at least one $\phi_t^\ell$ contained in $\omega_j$; recall here that $\Phi_\ell = \bigcup_{t \in [m-r]} \phi_t^\ell \times \{t\}$ is the $(r,m)$-SLMF induced by the relaxed $(1,r,m)$-SLMF $\Omega_{\mathcal J_\ell}$. By Lemma \ref{lem:NonDimensionDrop} the $k$-vector space $\pi_{\phi_{t}^\ell}(S)$ has dimension equal to $r$. Thus $\dim_k \pi_{\omega_j}(S)=r$. Since $\pi_{\omega_j}(x_j) = \pi_\omega(x_j')$, Lemma \ref{lem:x-reconstruction} gives $x_j = x_j'$. This is true for every $j$ and so $X=X'$. We conclude that $X$ is the only closed point of the fiber $\pi_\Omega'^{-1} \big( \pi'_\Omega(X) \big)$. As a scheme, the fiber is the spectrum of a finitely generated $k$-algebra $R$; such a ring is a Jacobson ring, that is every prime ideal of $R$ is the intersection of maximal ideals (e.g., Theorem 4.19 in \cite{Eisenbud-2004}). But as we just saw, $R$ has a single maximal ideal and so $\pi_\Omega'^{-1} \big( \pi'_\Omega(X) \big)$ is a $0$-dimensional scheme, which as a topological space consists of a single point $X$. 

Now $\pi_\Omega$ is a morphism of integral schemes. Since an open subscheme of an integral subscheme is integral, $\pi_\Omega'$ is also a morphism of integral schemes. Since $\dim U_\Omega = \dim \mathbb{A}^{\Omega} = \dim \M(r, m \times n)$ and $\dim \pi_\Omega'^{-1} \big( \pi'_\Omega(X) \big) = 0$, the upper-semicontinuity of the fiber dimension on the source for morphisms of finite-type $k$-schemes shows that $\pi_\Omega'$ is dominant. But then a fortiori $\pi_\Omega$ is dominant, whence $\Omega$ is a base set of the algebraic matroid of $k[Z]/I_{r+1}(Z)$ (recall the discussion in section \ref{subsection:TheAlgebraicMatroid}).
%%%%%%%
\subsection{Proof of Proposition \ref{prp:necessary}} \label{subsection:proof-prp:necessary}

We begin with some reductions. Since $\Omega$ is a base set, Lemma \ref{lem:less-than-r} gives $\#\omega_j \ge r$ for every $j \in [n]$. Neither the hypothesis nor the conclusion of the proposition is affected if we remove the $\omega_j$'s of cardinality $r$; for the former we refer to Lemma \ref{lem:omega-r}, for the latter we observe that the condition of a relaxed $(\nu,r,m)$-SLMF is insensitive to the presence of such $\omega_j$'s. We will thus assume $\# \omega_j \ge r+1$ for every $j \in [n]$. By Lemma \ref{lem:k->K} we assume $k$ is algebraically closed. 

We identify convenient open sets to work over. We let $U_1$ be the open set of $\M(r, m \times n)$ on which every closed point $X$ has column-space of dimension $r$, and so does the projection of its column-space onto any $r$ coordinates of $k^m$. For any $X=[x_1 \cdots x_n] \in  U_1$ there is a dense open set $V_X \subset \Gr(r,m)$ with the following property: for any $S \in V_X$ and any $k$-basis $B \in k^{m \times r}$ of $S$, and for every $j \in [n]$ and any $\psi_j \subset \omega_j$ with $\#\psi_j = r$, i) the matrix $\pi_{\psi_j}(B) \in k^{r \times r}$ is invertible, and ii) the matrix 
\begin{align}
X_S:=[B \pi_{\psi_1}(B)^{-1} \pi_{\psi_1}(x_1) \cdots B \pi_{\psi_n}(B)^{-1} \pi_{\psi_n}(x_n)] \in k^{m \times n} \label{eq:X_S}
\end{align}

\noindent has rank $r$; that $V_X$ is open can be checked on the standard affine open cover of $\Gr(r,m)$, that it is non-empty follows because Lemma \ref{lem:x-reconstruction}) gives $X_{\mathfrak{c}(X)}= X$, where $\mathfrak{c}(X)$ is the column-space of $X$, so $\mathfrak{c}(X) \in V_X$.

We verify the relaxed $(r,r,m)$-SLMF condition for $\mathcal I = [m]$. Since $\Omega$ is a base set, $\#\Omega=\sum_{j \in [n]} \# \omega_j = r(m+n-r)$. With this, since $\# \omega_j \ge r$ for every $j \in [n]$, we have $\sum_{j \in [n]} \max\{\# \omega_j-r,0\} =\sum_{j \in [n]} (\# \omega_j-r)= r(m-r) = r(\# \mathcal I -r)$, where $\mathcal I = [m]$. 

Next, we define certain hyperplane sections of $\Gr(r,m)$. Since $\Omega$ is a base set, there is a dense open set $ U_2 \subset \M(r, m \times n)$ such that the fiber $\pi_\Omega^{-1}(\pi_\Omega(X))$ is $0$-dimensional for any $X \in U_2$. Pick an $X=[x_1 \cdots x_n] \in U_1 \cap U_2$. Denote by $x_{ij}$ the $i$th coordinate of $x_j$. For any $j \in [n]$ fix a $\psi_j \subset \omega_j$ with $\#\psi_j=r$ and for every $t \in \omega_j \setminus \psi_j$ let $\phi_{j,t}=\psi_j \cup \{t\} = \{i_1,\dots,i_{r+1}\} \in \Omega_j$. Recalling the notation of section \ref{subsection:Conventions}, we define a linear form in Pl{\"u}cker coordinates
$$ l_{j,\psi_j,t} = \sum_{\alpha \in [r+1]} \, (-1)^{\alpha+1} \, x_{i_\alpha j} \, [\phi_{j,t} \setminus \{i_\alpha\}] \, \in \,  T. $$

\noindent We let $ L_\Omega $ be the ideal of $ T$ generated by the $l_{j,\psi_j,t}$'s for all $j$'s and $t$'s, and we note that there are $\sum_{j \in [n]} \max\{\# \omega_j-r,0\} =\sum_{j \in [n]} (\# \omega_j-r)= r(m-r) = \dim \Gr(r,m)$ in total $l_{j,\psi_j,k}$'s regardless of how we choose the $\psi_j$'s. 

Now we show that the linear section of $\Gr(r,m)$ we just defined cuts $V_X$ to a zero-dimensional locus. Let $ V_X'$ be any basic open subset of $V_X$, say $V_X' = \Spec (T/\p)_{(\overline{p})} $ for some homogeneous polynomial $p \in T$, with $\overline{p}$ its class in $ T/\p$ and $( T/\p)_{(\overline{p})}$ the homogeneous localization of $ T/\p$ at the multiplicatively closed set $\{1,\overline{p},\overline{p}^2,\dots\}$. In view of Lemma \ref{lem:sections}, every closed point of $\Spec (T/\p+L_\Omega)_{(\overline{p})} $ is an $r$-dimensional linear subspace $S$ of $k^m$ for which $\pi_{\phi_{j,t}}(x_j) \in \pi_{\phi_{j,t}}(S)$ for every $j \in [n]$ and every $t \in \omega_j \setminus \psi_j$. In particular, $\pi_{\psi_j}(x_j) \in \pi_{\psi_j}(S)$ for every $j \in [n]$. By definition of $V_X$ we have $\dim_{k} \pi_{\psi_j}(S)=r$ for every $j \in [n]$. Thus by Lemma \ref{lem:x-reconstruction} and its proof, the matrix $X_S$ in \eqref{eq:X_S} is the unique matrix $[v_1 \cdots v_n] \in k^{m \times n}$ with columns in $S$ that satisfies $\pi_{\psi_j}(v_j)=\pi_{\psi_j}(x_j)$ for every $j \in [n]$. Now, the relations $\pi_{\phi_{j,t}}(x_j) \in \pi_{\phi_{j,t}}(S)$ and Lemma \ref{lem:x-reconstruction-union} imply that $\pi_{\omega_j}(x_j) \in \pi_{\omega_j}(S)$ for every $j \in [n]$. Then again by Lemma \ref{lem:x-reconstruction}, there is a unique matrix $[w_1 \cdots w_n] \in k^{m \times n}$ with columns in $S$ that satisfies $\pi_{\omega_j}(w_j) =\pi_{\omega_j}(x_j)$ for every $j \in [n]$. But note that this matrix also satisfies $\pi_{\psi_j}(w_j)=\pi_{\psi_j}(x_j)$, and by the said uniqueness it must be equal to $X_S$. This shows that $S$ induces a unique rank-$r$ completion $X_S \in U_1 \cap U_2$ of $\pi_\Omega(X)$, and in turn $X_S$ uniquely determines $S$ by the relation $S = \mathfrak{c}(X_S)$. Since by hypothesis $\pi_\Omega^{-1}(\pi_\Omega(X))$ is a finite set, we must have $\dim \Spec ( T/\p+ L_\Omega)_{(\overline{p})} =0$ or equivalently the Krull dimension of $( T/\p+ L_\Omega)_{(\overline{p})}$ is zero. 

Suppose that $\Omega$ violates the relaxed $(r,r,m)$-SLMF condition for some $\mathcal I \subsetneq [m]$ with $\# \mathcal I \ge r+1$. Then $\sum_{j \in [n]} \max\{\#(\omega_j \cap \mathcal I) -r,0 \}> r(\#\mathcal I-r)$. With $\mathcal J$ the set of $j$'s for which $\#(\omega_j \cap \mathcal I) -r > 0$, we equivalently have $\sum_{j \in \mathcal J} (\#(\omega_j \cap \mathcal I) -r)> r(\#\mathcal I-r)$. For every $j \in \mathcal J$ let us select the $\psi_j$ associated to the construction of the $l_{j,\psi_j,t}$'s to be a subset of $\mathcal I$. Then this inequality says that among a total of $r(m-r)$ linear forms generating $L_\Omega $ there are more than  $r(\#\mathcal I-r)$ of them that are supported on Pl{\"u}cker coordinates with indices in $\mathcal I$. Fix any $j \in \J$ and set $\psi=\psi_j$. By definition of $ V_X$ and $V_X'$, we may assume that $[\psi]$ divides $p$. Thus $( T/\p+ L_\Omega)_{(\overline{p})} = \big(( T/\p+ L_\Omega)_{(\overline{[\psi]})}\big)_{(\overline{p})}$. Now $( T/\p)_{(\overline{[\psi]})}$ is just a polynomial ring over $k$ of dimension $r(m-r)$, because it is the coordinate ring of $\Gr(r,m)$ at the standard open set $[\psi] \neq 0$. The variables of this polynomial ring can be thought organized in the entries corresponding to rows indexed by $[m] \setminus \psi$, of an $m\times r$ matrix $B$ whose $r \times r$ submatrix indexed by rows $\psi$ is the identity. Hence $( T/\p+ L_\Omega)_{(\overline{p})}$ is the localization of the quotient of a polynomial ring $k[B]$ of dimension $r(m-r)$ by an ideal $ J$ generated by $r(m-r)$ elements; let us call this $(k[B] /  J)_{\overline{f}}$. To obtain the generators of $ J$, we replace each variable $[\psi'] \in  T$ by the $r\times r$ determinant of $B$ associated to row indices $\psi'$. As $\psi'$ ranges among all subsets $\mathcal I$ of size $r$, this substitution yields among a total of $r(m-r)$ generators of $ J$ more than $r(\#\mathcal I-r)$ of them involving at most $r(\# \mathcal I-r)$ variables of $k[B]$. Let $ J_1$ be the ideal of $k[B]$ generated by this subset of generators of $ {J}$, and let $ J_2$ be the ideal generated by the remaining generators of $ J$. Since the generators of $J_1$ live in a polynomial ring of dimension $r(\# \mathcal I-r)$, every minimal prime $P$ over $J_1$ must have height at most $r(\# \mathcal I-r)$. Since $J_2$ has less than $r(m-r) - r(\# \mathcal I-r)$ generators, Krull's height theorem gives that every minimal prime $Q$ of $J_2$ has height less than $r(m-r) - r(\# \mathcal I-r)$. Now, every minimal prime $H$ of $J= J_1 +  J_2$ is necessarily a minimal prime over $P+Q$ for some $P$ and $Q$ as above, and by intersection theory in affine space (Proposition I.7.1 in \cite{hartshorne1977algebraic} or Chapter III Proposition 17 in \cite{serre2012local}, or even more generally Chapter V Theorem 3 in \cite{serre2012local}) $H$ has height at most $\height(P) + \height (Q)< r(m-r)$. Hence, every irreducible component of $\Spec k[B] /  J$ has positive dimension, contradicting the fact that the Krull dimension of $(k[B] /  J)_ {\overline{f}} $ is zero. Thus $\Omega$ must satisfy the condition of a relaxed $(r,r,m)$-SLMF for every $\mathcal I \subsetneq [m]$ with $\# \mathcal I \ge r+1$. 

%%%%%%%%
\subsection{Proof of Proposition \ref{prp:conj}} \label{subsection:proof-prp:conj}

The \emph{if} part follows from Theorem \ref{thm:main} and Proposition \ref{prp:necessary} so that only the \emph{only if} part needs proving. By a similar argument as in the second paragraph of the proof of Theorem \ref{thm:main}, the relaxed $(r,r,m)$-SLMF condition implies $\#\omega_j \ge r$ for every $j \in [n]$. Since the $(\nu,r,m)$-SLMF condition is insensitive to the presence of $\omega_j$'s with cardinality $r$, we will assume that  $\# \omega_j \ge r+1$ for every $j \in [n]$; after this reduction it may no longer be true that $m \le n$.  

\emph{i) $\underline{r=1}$.} For $r=1$ the statement is trivially true. 

\emph{ii) $\underline{r=m-1}$.}  By hypothesis, $\omega_j = [m], \, \forall j \in [n]$. Thus $\#\Omega = mn$. On the other hand, $\#\Omega = \dim \M(m-1, m \times n) = (m-1)(n+1)$. Thus $mn = (m-1)(n+1)$, which gives $n = m-1=r$. Taking $\J_\ell = \{\ell\}$ for every $\ell \in [n]=[m-1]$ gives the required partition. 

\emph{iii) $\underline{r=m-2}$.} When $r = m-2$, we have $m \ge 3$, since by Notation \ref{not:rmn} $r$ is always positive. Since $\#\omega_j \ge r+1$ for every $j \in [n]$, each $\omega_j$ is either equal to $[m]$ or has cardinality $m-1$. Without loss of generality we assume $\omega_j = [m]$ for $j=n-\alpha+1,\cdots,n$, and $\# \omega_j = m-1$ for $j \in [n-\alpha]$, for some non-negative integer $\alpha$. As in the previous case, from $\#\Omega = (n-\alpha)(m-1)+\alpha m = (m-2)(n+2)$, we have $n = 2m-4-\alpha$. 

We construct the partition $[n] = \bigcup_{\ell \in [m-2]} \J_\ell$ as follows. Since $[m]$ is a relaxed $(1,m-2,m)$-SLMF, we set $\J_\ell = \{\ell\}$ for $\ell = n-\alpha+1,\dots,n$. It remains to partition $\Omega_{[n-\alpha]} = \Omega_{[2m-4-2\alpha]}$ into $m-2-\alpha$ relaxed $(1,m-2,m)$-SLMF's $\Omega_{\J_1},\dots,\Omega_{\J_{m-2-\alpha}}$. For $j \in [2m-4-2\alpha]$ we re-order the $\omega_j$'s so that equal $\omega_j$'s are placed consecutively. For $\ell \in [m-2-\alpha]$ we define $\J_\ell = \{\ell, \ell + m-2-\alpha\}$. The only way that some $\Omega_{\J_\ell}$ is not a relaxed $(1,m-2,m)$-SLMF is if $\omega_\ell = \omega_{\ell+m-2-\alpha}$. By the way we have ordered the $\omega_j$'s, necessarily $\omega_\ell = \omega_{\ell+1} = \cdots = \omega_{\ell+m-3-\alpha} =\omega_{\ell+m-2-\alpha}$. Then for $\I = \omega_\ell$ we have
\begin{align}
\sum_{j \in [n]} \big(\#(\omega_j \cap \I) - (m-2)\big) &= \sum_{j \in [n-\alpha]} \big(\#(\omega_j \cap \I) - (m-2)\big) + \sum_{j \ge n-\alpha+1} \big(\#(\omega_j \cap \I) - (m-2)\big) \nonumber \\
& \ge \big((m-2-\alpha)+1\big) + \alpha = m-2+1 > r (\# \I -r) = m-2. \nonumber
\end{align} This violates the hypothesis of relaxed $(m-2,m-2,m)$-SLMF on $\Omega$.

%%%%%%%%
\subsection{Proof of Proposition \ref{prp:Amini}} \label{subsection:proof-prp:Amini}
\emph{Step 1.} Define a function $f:2^{[n]} \rightarrow \mathbb{N}$ by $f(\J) = \# \bigcup_{j \in \J} \omega_j -r$ for every non-empty $\J \subset [n]$ and $f(\emptyset)=0$. One easily sees that $f$ has the so-called \emph{intersection-submodular} property, i.e. it satisfies the submodular inequality $f(\J_1) +f(\J_2) \ge f(\J_1 \cup \J_2) + f(\J_1 \cap \J_2)$, for any $\J_1, \J_2 \subset [n]$ with $\J_1 \cap \J_2 \neq \emptyset$. Now, by Theorems 2.5 and 2.6 in \cite{fujishige2005submodular}, $f$ induces a submodular function $\hat{f}:2^{[n]} \rightarrow \mathbb{N}$, the so-called \emph{Dilworth-truncation} of $f$, given by the formula $\hat{f}(\J) = \min_{\J_1,\cdots,\J_s} \sum_{i \in [s]} f(\J_i)$, where the collection of $\J_i, \, i \in [s]$ varies across all possible partitions of $\J$. In addition to being submodular, one readily checks that $\hat{f}(\emptyset)=0$, $\hat{f}$ is non-decreasing, that is $\hat{f}(\J_2) \le \hat{f} (\J_1)$ for any $\J_2 \subset \J_1$, and also $\hat{f}(j) = 1$ for every $j \in [n]$. This makes $\hat{f}$ into the rank function of a matroid $\mathcal{M}$ on $[n]$. 

\emph{Step 2.} The rank of $\mathcal M$ is equal to $\hat{f}([n])$ and we now show that this value is $m-r$. First, note that an application of the relaxed $(r,r,m)$-SLMF condition for $\I = [m]$, together with the hypothesis $\#\omega_j = r+1, \, \forall j \in [n]$, yields that $n = r(m-r)$. Now, if $\I = \bigcup_{j \in [n]} \omega_j$ is not equal to $[m]$, another application of the same condition with this $\I$ gives the contradiction $n \le r(\# \I -r) < r(m-r)$. So $\bigcup_{j \in [n]} \omega_j = [m]$ and thus $\hat{f}([n]) \le f([n]) = m-r$. Suppose there is some partition $\J_i, \, i \in [s]$, of $[n]$ with $s >1$ such that  
$\hat{f}([n]) = \sum_{i \in [s]} f(\J_i) < m-r.$ For every $\J_i$ define $\I_i = \bigcup_{j \in \J_i} \omega_j$, then the inequality can be written as $\sum_{i \in [s]} (\#\I_i-r)< m-r$. From the $(r,r,m)$-SLMF condition applied to $\I_i$ we get 
$\# \J_i \le \sum_{j \in [n]} \max\{\#(\omega_j \cap \I_i)-r,0\} \le r(\# \I_i -r)$. Combining all the above, we arrive at the contradiction $r(m-r)=n=\sum_{i \in [s]} \# \J_i \le \sum_{i \in [s]} r(\# \I_i -r) < r(m-r)$. 

\emph{Step 3.} We show that $\J \subset [n]$ is a base of $\mathcal M$ if and only if $\Omega_\J$ is a relaxed $(1,r,m)$-SLMF. Under the hypothesis $\#\omega_j = r+1, \, \forall j \in [n]$, by Lemma \ref{lem:SLMF-I} we have that $\Omega_\J$ is a relaxed $(1,r,m)$-SLMF if and only if it is an $(r,m)$-SLMF. Suppose $\J \subset [n]$ is a base of $\mathcal M$. Then $m-r = \rank \mathcal{M} = \#\J =\hat{f}(\J) \le f(\J) =  \bigcup_{j \in \J} \omega_j -r \le m-r$. Thus equality holds everywhere and $\Omega_{\J}$ satisfies the equality required in \eqref{eq:SLMF}. Let $\J'$ be any subset of $\J$. Since $\J$ is independent, so is $\J'$. Thus $\hat{f}(\J') = \# \J'$, because the rank of an independent set in a matroid equals the size of the set. Since by definition $\hat{f}(\J') \le f(\J')$, we also have $\# \J' \le f(\J') = \# \bigcup_{j \in \J'} \omega_j -r$. Thus $\Omega_{\J}$ also satisfies the inequalities required in \eqref{eq:SLMF} and so it is an $(r,m)$-SLMF. Conversely, a similar argument as in Step 2 shows that if $\Omega_\J$ is an $(r,m)$-SLMF, then $\hat{f}(\J) = m-r$. 

\emph{Step 4.} It is enough to show the existence of $r$ disjoint bases of $\mathcal{M}$. This follows immediately from a theorem of Edmonds \& Fulkerson together with the definition of relaxed $(r,r,m)$-SLMF. Indeed, an application of Theorem 2b in \cite{edmonds1965transversals}, shows that such bases exist if and only if for every $\J \subset [n]$ we have $\# \J \ge r(m-r)-r \, \hat{f}([n]\setminus \J)$. Now let $J_i, \, i \in [s]$, be a partition of $[n]\setminus \J$ such that $\hat{f}([n]\setminus \J) = \sum_{i \in [s]} f(J_i)$. Then we have to show that $r(m-r) \le \#\J + \sum_{i \in [s]} r \, f(\J_i)$. This follows by noting, as in Step 2, that the relaxed $(r,r,m)$-SLMF condition implies $\#\J_i \le r \, f(\J_i)$, while from Step 1 we have $ n = r(m-r)$, and of course $\#\J+ \sum_{i \in [s]} \# \J_i = n$.

%%%%%%%
\subsection{Proof of Proposition \ref{prp:injective}} \label{subsection:proof-prp:injective}

 %Let $Z_{[r],[r]}$ be the top leftmost $r \times r$ submatrix of $Z$. With $p= \det (Z_{[r],[r]})$, the localization $(k[Z] / I_{r+1}(Z))_p$ of the ring $k[Z] / I_{r+1}(Z)$ at the multiplicatively closed set $\{1,\bar{p},\bar{p}^2,\dots\}$ is the localization of the polynomial ring $k[z_{ij}: \min\{i,j\} < r+1]$ at the same multiplicatively closed set. From this, we see that $\M(r, m \times n)$ is birational to the affine space of dimension $r(m+n-r)$ over $k$, and as such, it is rational. We denote by $\M(r, m \times n)(k)$ be the set of $k$-rational points of $\M(r, m \times n)$; since $\M(r, m \times n)$ is a rational variety, $\M(r, m \times n)(k)$ is dense. 

\emph{Step 1.} We define a convenient dense open locus $U_1 \subset \M(r,m \times n)$. Write $\Phi_\ell = \Phi = \bigcup_{j \in [m-r]} \phi_j \times \{j\}$ for every $\ell$. By hypothesis, for every $\alpha \in [m-r]$ there is a subset $\mathcal{L}_\alpha \subseteq [n]$ of size $r$ such that $\phi_\alpha \subseteq \omega_{j}, \, \forall j \in \mathcal{L}_\alpha$. For a $k$-rational point $X=[x_1 \cdots x_n] \in \M(r, m \times n)$ denote by $\mathfrak{c}(X)$ the column-space of $X$. Call $U_1$ the non-empty open set of $\M(r,m \times n)$ on which i) $\mathfrak{c}(X)$ lies in $V_\Phi$ (section \ref{subsection:SLMFs}), ii) none of the Pl{{\"u}}cker coordinates of $\mathfrak{c}(X)$ vanishes, iii) the $\{x_j: \,  j \in \mathcal{L}_\alpha\}$ are a $k$-basis for $\mathfrak{c}(X)$ for every $\alpha \in [m-r]$, and iv) $\pi_\Omega^{-1}(\pi_\Omega(X))$ is zero-dimensional (possible because by Theorem \ref{thm:main} $\Omega$ is a base set). Since $\Span(x_j: \, j \in \mathcal{L}_\alpha)$ is the same as $\mathfrak{c}(X)$, we also have $\Span(\pi_{\phi_\alpha}(x_j): \, j \in \mathcal{L}_\alpha) = \pi_{\phi_\alpha}(\mathfrak{c}(X))$. Proposition \ref{prp:SZ93} asserts that the vector spaces $\pi_{\phi_\alpha}\big(\mathfrak{c}(X)\big), \, \alpha \in [m-r]$, uniquely determine $\mathfrak{c}(X)$ on $V_\Phi$. Since all $\pi_{\phi_\alpha}(x_j)$'s can be extracted from $\pi_\Omega(X)$, we have that $\pi_\Omega(X)$ uniquely determines $\mathfrak{c}(X)$ on $V_\Phi$. Since $\# \omega_j \ge r$ (we proved this in the proof of Theorem \ref{subsection:Proof-main}) and $\pi_{\omega_j}\big(\mathfrak{c}(X)\big)$ does not drop dimension for any $j \in [n]$, Lemma \ref{lem:x-reconstruction} gives that the data $\mathfrak{c}(X), \, \pi_\Omega(X)$ uniquely determine $X$. Hence, $\pi_\Omega(X)$ uniquely determines $X$ on $U_1$ for any $k$-rational $X \in U_1$, that is $X$ is the only $k$-rational point of $\pi_\Omega^{-1}(\pi_\Omega(X)) \cap U_1$. 

\emph{Step 2.} For any $(i,j) \not\in \Omega$ we define a polynomial $f_{ij}$. Since $\Omega$ is a base set, $\Omega \cup \{(i,j)\}$ is a dependent set in the algebraic matroid of $\M(r,m \times n)$.
There is a unique circuit $C_{ij}$ contained in $\Omega \cup \{(i,j)\}$; otherwise, if $C_1, C_2$ are distinct circuits of $\Omega \cup \{(i,j)\}$, then necessarily $(i,j) \in C_1 \cap C_2$ and by a fundamental property of circuits (e.g., see part (3) of Definition 6 in \cite{rosen2020algebraic}) there must be a circuit $C_3$ contained in $C_1 \cup C_2 \setminus \{(i,j)\}$, a contradiction on the independence of $\Omega$. Let $Z_{C_{ij}} = \{z_{\alpha \beta}: \, (\alpha,\beta) \in C_{ij}\}$. The ideal $I_{r+1}(Z) \cap k[Z_{C_{ij}}]$ of $k[Z_{C_{ij}}]$ is prime, because it is the kernel of the map $k[Z_{C_{ij}}] \rightarrow k[Z] / I_{r+1}(Z)$ and the target ring is an integral domain. Thus $k[Z_{C_{ij}}] / I_{r+1}(Z) \cap k[Z_{C_{ij}}]$ is an integral domain of Krull dimension $\# C_{ij} - 1$. Then for dimension reasons, $I_{r+1}(Z) \cap k[Z_{C_{ij}}]$ must be a principal ideal generated by an irreducible polynomial that we call $f_{ij}$. By elimination theory, $f_{ij}$ can be determined from a Gr{\"o}bner basis computation as follows. Equip $k[Z]$ with a lexicographic order in which the variables $Z_{C_{ij}}$ are least significant. Then $f_{ij}$ is the unique member in $k[Z_{C_{ij}}]$ of a reduced Gr{\"o}bner basis of $I_{r+1}(Z)$ in the above lexicographic order. This Gr{\"o}bner basis can be obtained by applying Buchberger's algorithm in the $(r+1)$-minors of $Z$. Suppose $\operatorname{char} k = 0$. Then all computations in Buchberger's algorithm will take place in the polynomial ring $\mathbb{Q}[Z]$ to furnish $f_{ij}$. If $\operatorname{char}k = p>0$, all computations will take place in the polynomial ring $\mathbb{Z}_p[Z]$, and for large enough $p$ they will yield the same polynomial $f_{ij}$ as in zero characteristic. We conclude that the polynomials $f_{ij}, \, (i,j) \not\in \Omega$, will not depend on $k$. Moreover, $f_{ij}$ remains irreducible over $\overline{k}[Z_{C_{ij}}]$, where $\overline{k}$ is the algebraic closure of $k$, since it still generates a prime ideal.

\emph{Step 3.} Let $d_{ij}$ be the degree of $z_{ij}$ in $f_{ij}$; we will prove that $d_{ij}=1$. 

We first define an open locus $U_{ij} \subset \M(r, m \times n)$. By Step 2, the $f_{ij}$'s do not depend on $k$ (if $\operatorname{char} k >0$ is large enough); we thus assume that $k = \overline{k}$ and if $\operatorname{char}k>0$, we assume that it is larger than every $d_{ij}$. Let $R({f_{ij},f_{ij}'}) \in k[Z_{C_{ij}\setminus \{(i,j)\}}]$ be the resultant of $f_{ij}$ and its derivative $f_{ij}'$, viewed as polynomials in $z_{ij}$ with coefficients in $k[Z_{C_{ij}\setminus \{(i,j)\}}]$ (see section IV.8 in \cite{lang2002graduate} for a description and properties of the resultant and its relation to the discriminant). Suppose $R({f_{ij},f_{ij}'}) = 0$. Viewing $f_{ij}, f_{ij}'$ as polynomials with coefficients in the field of fractions of $k[Z_{C_{ij}\setminus \{(i,j)\}}]$, which we will denote by $\mathscr{K}$, it follows from a basic property of the resultant that $f_{ij}$ and $f_{ij}'$ have a common root $\mathfrak{r}$ in the algebraic closure $\overline{\mathscr{K}}$ of $\mathscr{K}$. Hence, both of them are divisible in $\mathscr{K}[z_{ij}]$ by the minimal polynomial $p_{\mathfrak{r}} \in \mathscr{K}[z_{ij}]$ of $\mathfrak{r}$. This is a contradiction because $f_{ij}$ is irreducible in $k[Z_{C_{ij}\setminus \{(i,j)\}}][z_{ij}]$ and by Gauss's lemma it remains irreducible in $\mathscr{K}[z_{ij}]$. Thus $R({f_{ij},f_{ij}'})$ is a non-zero polynomial and so the discriminant $\Delta({f_{ij},f_{ij}'}) \in k[Z_{C_{ij}\setminus \{(i,j)\}}]$ is also non-zero. Since $\Delta({f_{ij},f_{ij}'})$ is a non-zero polynomial in algebraically independent variables modulo $I_{r+1}(Z)$, the open locus $U_{ij} = \Spec (k[Z]/I_{r+1}(Z))_{\Delta_{ij}}$ is dense. For every closed point $X=(x_{\alpha \beta}) \in U_{ij}$ (which is $k$-rational since $k$ is algebraically closed), the substitution $z_{\alpha \beta} \mapsto x_{\alpha \beta}$ in $f_{ij}$ for $(\alpha, \beta) \in C_{ij} \setminus \{(i,j)\}$ gives a polynomial $f_{ij}(X)$ of degree $d_{ij}$ in $k[z_{ij}]$, which has non-zero discriminant and thus $d_{ij}$ distinct roots in $k$.

Next, we refine the open locus $U_1 \cap U_{ij} $. With $\Omega' = \Omega \cup \{(i,j)\}$,  consider the factorization $\pi_\Omega : \M(r, m \times n) \stackrel{\pi_{\Omega'}}{\longrightarrow} \mathbb{A}^{\Omega'} \stackrel{\pi}{\longrightarrow} \mathbb{A}^\Omega$, where $\pi$ drops the coordinate associated to $(i,j)$. The closure of the image of $\pi_{\Omega'}$ is the hypersurface of $\mathbb{A}^{\Omega'}$ defined by $f_{ij}$; we will denote it by $\mathbb{V}(f_{ij})$. We define two closed sets $W_1$ and $W_2$ of $\mathbb{A}^{\Omega'}$. With $Y = \M(r,m\times n) \setminus U_1 \cap U_{ij}$, the irreducibility of $\M(r,m\times n)$ gives $\dim Y < \dim \M(r, m \times n) = \# \Omega$. Thus we also have $\dim W_1 = \dim \pi_{\Omega'}(Y) < \dim \mathbb{V}(f_{ij}) = \# \Omega$, where $W_1$ is the closure of $\pi_{\Omega'}(Y)$ in $\mathbb{V}(f_{ij})$. Next, the image of $U_1 \cap U_{ij}$ under $\pi_{\Omega'}$ is dense in $\mathbb{V}(f_{ij})$; this is because the dimension of $U_1 \cap U_{ij}$ is $\#\Omega$ and the generic fiber of $\pi_{\Omega'}|_{U_1 \cap U_{ij}}$ is finite. By Chevalley's theorem, $\pi_{\Omega'}(U_1 \cap U_{ij})$ is constructible and so it contains a dense open set $V$ of $\mathbb{V}(f_{ij})$. We set $W_2 = \mathbb{V}(f_{ij})\setminus V$. Since $\mathbb{V}(f_{ij})$ is irreducible (it is defined by the irreducible polynomial $f_{ij}$), $\dim W_2 < \#\Omega$. We have that $W=W_1 \cup W_2$ is a closed subset of $\mathbb{V}(f_{ij})$ with $\dim W < \#\Omega$; hence also $\dim(\pi(W))< \#\Omega$. With $O$ the complement of the closure in $\mathbb{A}^{\Omega}$ of $\pi(W)$, we have that $\pi_\Omega^{-1}(O)$ is a dense open set of $\M(r, m \times n)$ contained in $U_1 \cap U_{ij}$. 

Take any closed point $X \in \pi_\Omega^{-1}(O)$; $X$ lies in $U_{ij}$, so $f_{ij}(X)$ has $d_{ij}$ distinct roots including $x_{ij}$. Suppose $d_{ij}>1$. Let $x$ be a root of $f_{ij}(X)$ different from $x_{ij}$; let $\pi_{\Omega'}(X)^*$ be the $m \times n$ matrix obtained from $\pi_{\Omega'}(X)$ by replacing $x_{ij}$ with $x$ at position $(i,j)$. Suppose that $\pi_{\Omega'}(X)^*$ is not in the image of $\pi_{\Omega'}$. However, $\pi_{\Omega'}(X)^*$ is in the closure of the image of $\pi_{\Omega'}$ because it is a root of $f_{ij}$. Then by construction $\pi_{\Omega'}(X)^* \in W$, whence $\pi_\Omega(X) \in \pi(W)$; this contradicts the choice of $X$. Hence $\pi_{\Omega'}(X)^* \in \im(\pi_{\Omega'})$ for each of the $d_{ij}$ distinct roots of $f_{ij}(X)$. We claim more, that $\pi_{\Omega'}(X)^*$ is in the image of $U_1$; for if not, then $\pi_{\Omega'}(X)^*$ is in the image of $Y$, and so $\pi_{\Omega'}(X)^* \in W$ whence $\pi_\Omega(X) \in \pi(W)$, again contradicting the choice of $X$. We conclude that $\pi_\Omega^{-1}(\pi_\Omega(X)) \cap U_1$ contains $d_{ij}$ distinct elements; these are all closed points (and thus $k$-rational) because $\pi_\Omega^{-1}(\pi_\Omega(X))$ is zero-dimensional. But by what we proved about $U_1$ in Step 1, $X$ is the only closed point of $\pi_\Omega^{-1}(\pi_\Omega(X))$; this is a contradiction on the hypothesis that $d_{ij}>1$. 

\emph{Step 4.} For each $(i,j) \not\in \Omega$ write $f_{ij} = g_{ij} z_{ij}-h_{ij} \in I_{r+1}(Z)$ with $g_{ij},h_{ij} \in k[Z_\Omega]$. Denote by $g$ the product of all $g_{ij}$'s, and let $ U_\Omega=\Spec(k[Z]/I_{r+1}(Z))_g$ be the dense open locus of $\M(r, m \times n)$ where $g$ does not vanish. Take a prime ideal $Q \in \Spec(k[Z]/I_{r+1}(Z))_g$ and let $P$ be the contraction of $Q$ under the ring homomorphism $k[Z_\Omega] \rightarrow (k[Z]/I_{r+1}(Z))_g$. Then the fiber $\pi_\Omega^{-1}(P)$ is the spectrum of the fiber ring $R=k[Z] / I_{r+1}(Z) \otimes_{k[Z_\Omega]} \kappa(P)$, where $\kappa(P)$ is the residue class field of $P$. Note $z_{ij} = h_{ij} / g_{ij}$ in $(k[Z] / I_{r+1})_g$; thus $(k[Z] / I_{r+1})_g = k[Z_\Omega]_g / J$ for some ideal $J$ of $k[Z_\Omega]_g$. Since we have an equality of Krull dimensions $\dim (k[Z] / I_{r+1})_g = \dim k[Z_\Omega] = \dim k[Z_\Omega]_g$ and $k[Z_\Omega]_g$ is an integral domain, $J=0$ and so $(k[Z] / I_{r+1})_g = k[Z_\Omega]_g$. Hence $$R = k[Z] / I_{r+1}(Z) \otimes_{k[Z_\Omega]} \kappa(P) =(k[Z] / I_{r+1}(Z))_g \otimes_{k[Z_\Omega]_g} \kappa(P)= \kappa(P),$$ and the fiber consists of a single point $Q$.

%%%%%%%%%%%%%
\section{Examples} \label{section:Examples}

\begin{example} \label{ex:SZ93}
Let $r=2, m=6$ and consider $\Phi=\bigcup_{j \in [4]} \phi_j \times \{j\} \subset [m] \times [m-r]$ with 
$$ \Phi = \{2,4,6\} \times \{1\} \, \, \cup \, \, \{1,2,4\} \times \{2\} \, \, \cup \, \, \{1,2,5\} \times \{3\} \, \, \cup \, \, \{1,3,5\} \times \{4\}  $$
and its representation by its indicator matrix:
$$ \begin{bmatrix}
0& 1 & 1 & 1  \\
1& 1 & 1 & 0 \\
0 & 0 & 0 & 1 \\
1 & 1 & 0 & 0 \\
0 & 0 & 1 & 1 \\
1& 0& 0& 0 
\end{bmatrix} $$ 

\noindent This is a $(2,6)$-SLMF since it satisfies condition \eqref{eq:SLMF}. It defines an open set $V_\Phi$ in $\Gr(2,6)$ on which the rational map $\Gr(2,6) \rightarrow \mathbb{P}^2 \times \mathbb{P}^2 \times \mathbb{P}^2 \times \mathbb{P}^2$ given by 
$$S \mapsto \left(\begin{bmatrix} \begin{array}{r} [46]_S\\ -[26]_S  \\ [24]_S \end{array} \end{bmatrix},  \begin{bmatrix} \begin{array}{r} [24]_S\\ -[14]_S \\ [12]_S \end{array}\end{bmatrix}, \begin{bmatrix} \begin{array}{r} [25]_S\\ -[15]_S \\ [12]_S \end{array} \end{bmatrix}, \begin{bmatrix} \begin{array}{r} [35]_S\\ -[15]_S \\ [13]_S \end{array} \end{bmatrix} \right)$$ 

\noindent is injective. These $4$ elements of $\mathbb{P}^2$ are precisely the normal vectors of the $4$ planes in $k^3$ that one gets by projecting a general $2$-dimensional subspace $S$ in $k^6$ onto the $3$ coordinates indicated by each of the $\phi_j$'s. 

For each $S \in V_\Phi$ and $B \in k^{6 \times 2}$ a $k$-basis of $S$, the following matrix is the unique up to a scaling of its columns $6 \times 4$ matrix supported on $\Phi$, whose columns span $S^\perp$:
$$ \begin{bmatrix}
\begin{array}{rrrr}
0 \, \, \, \, \, \, \,& [24]_S^B & [25]_S^B & [35]_S^B \\
[46]_S^B & -[14]_S^B & -[15]_S^B & 0 \, \, \, \, \, \, \, \\
0 \, \, \, \, \, \, \, & 0 \, \, \, \, \, \, \, & 0 \, \, \, \, \, \, \, & -[15]_S^B \\
-[26]_S^B  & [12]_S^B & 0 \, \, \, \, \, \, \, & 0 \, \, \, \, \, \, \, \\
0 \, \, \, \, \, \, \,& 0 \, \, \, \, \, \, \,& [12]_S^B & [13]_S^B \\
[24]_S^B & 0 \, \, \, \, \, \, \,& 0 \, \, \, \, \, \, \, & 0 \, \, \, \, \, \, \, \\
\end{array}
\end{bmatrix}$$ 
The polynomial that defines the complement of $V_\Phi$ is
$$p_\Phi = \det\left( 
\begin{bmatrix}
\begin{array}{rrrr}
[24]_S & [25]_S & [35]_S \\
0 \, \, \, \, \, \, \, & 0 \, \, \, \, \, \, \, & -[15]_S \\
0 \, \, \, \, \, \, \,& [12]_S & [13]_S \\
\end{array}
\end{bmatrix} 
\right) = [12]_S [24]_S [15]_S. $$ 
\end{example}

\begin{example} \label{ex:main-multiplePoints}
Let $r=2, m=6, n=8$ and $\Omega$ with $\#\Omega=24=\dim \M(2,6 \times 8)$ given by
$$\Omega = \begin{bmatrix}
1 & 0 & 0 & 1 & 0 & 1 & 1 & 1 \\
1 & 0 & 1 & 0 & 1 & 1 & 0 & 0 \\
1 & 0 & 1 & 0 & 1 & 0 & 1 & 0 \\
0 & 1 & 1 & 0 & 0 & 0 & 0 & 1 \\
0 & 1 & 0 & 1 & 1 & 0 & 0 & 0 \\
0 & 1 & 0 & 1 & 0 & 1 & 1 & 1
\end{bmatrix}.$$ With $\mathcal{T}_1=\{1,2,3,4\}$ and $\mathcal{T}_2=\{5,6,7,8\}$, the $\Phi_1$ and $\Phi_2$ above are the leftmost and rightmost respectively blocks of $\Omega$ and they both satisfy \eqref{eq:SLMF}. A computation with \texttt{Macaulay2} suggests that $\pi_\Omega^{-1}(\pi_\Omega(X))$ consists of $2$ points over a non-algebraically closed field and $4$ points otherwise. 
\end{example}

\begin{example} \label{ex:main}
Let $r=2, \, m=6, \, n=5$ and consider the following $\Omega \subset [6] \times [5]$ with $\# \Omega = 18 = \dim \M(2,6 \times 5)$ represented by its indicator matrix:
$$ \Omega = \begin{bmatrix} 
1 & 0 & 0 & 1 & 1 \\
1& 0 & 1 & 1 & 0 \\
1& 0 & 0& 0 & 1 \\
1 & 1 & 1 & 1 & 0 \\
1 & 1 & 0 & 1 & 1\\
0 & 1 & 0 & 1 & 0
\end{bmatrix} $$ Consider the partition $[5] = \mathcal{T}_1 \cup \mathcal{T}_2$ with $\mathcal{T}_1=\{1,2\}$ and $\mathcal{T}_2=\{3,4,5\}$. Now take 
$$ \Phi_1 = \begin{bmatrix} 
1 & 1 & 1 & 0 \\
1& 1 & 1 & 0 \\
1 & 0 & 0 & 0 \\
0 & 1 & 0 & 1 \\
0 & 0 & 1 & 1 \\
0& 0& 0& 1 
\end{bmatrix},  \, \, \, \Phi_2 = \begin{bmatrix} 
0& 1 & 1 & 1  \\
1& 1 & 1 & 0 \\
0 & 0 & 0 & 1 \\
1 & 1 & 0 & 0 \\
0 & 0 & 1 & 1 \\
1& 0& 0& 0 
\end{bmatrix}.$$ $\Phi_1$ and $\Phi_2$ are $(2,6)$-SLMF's since they satisfy \eqref{eq:SLMF}. $\Phi_1$ is associated with the first $2$ columns of $\Omega$ ($\mathcal{T}_1$), while $\Phi_2$ with the last $2$ columns of $\Omega $ ($\mathcal{T}_2$). A computation with \texttt{Macaulay2} suggests that $\pi_\Omega^{-1}(\pi_\Omega(X))$ consists only of $X$, for general $X$. 
\end{example}

\begin{example} \label{ex:flexibility}
This example demonstrates that Theorem \ref{thm:main} can be used with a certain flexibility in conjunction with Lemma \ref{lem:omega-r} and the operation of matrix transposition. 

It can be checked from Definition \ref{dfn:relaxedSLMF} that the following $\Omega \subset [6] \times [5]$, written in matrix indicator form, is a relaxed $(2,2,6)$-SLMF:
\begin{align*}
\Omega = 
\begin{bmatrix}
1 & 1 & 1 & 0 & 0 \\
1 & 0 & 1 & 0 & 0 \\
1 & 1 & 1 & 0 & 0 \\
1 & 0 & 0 & 1 & 1 \\
0 & 1 & 1 & 1 & 1 \\
0 & 1 & 0 & 1 & 1 
\end{bmatrix}
\end{align*} On the other hand, it is impossible to find a partition $[5] = \J_1 \cup \J_2$ such that both $\Omega_{\J_1}$ and $\Omega_{\J_2}$ are relaxed $(1,2,6)$-SLMF's; one of them will violate the relaxed $(1,2,6)$-SLMF condition for $\I = [6]$. But observe that the second row of $\Omega$ has only two $1$'s, i.e. not every vertex of $\mathcal{G}_\Omega$ has degree $\ge r+1$. Removing that row we obtain
\begin{align*}
\Omega' = 
\begin{bmatrix}
1 & 1 & 1 & 0 & 0 \\
1 & 1 & 1 & 0 & 0 \\
1 & 0 & 0 & 1 & 1 \\
0 & 1 & 1 & 1 & 1 \\
0 & 1 & 0 & 1 & 1 
\end{bmatrix}.
\end{align*} By Lemma \ref{lem:omega-r} applied on the matrix transpose of $\Omega$, we have that $\Omega$ is a base of the matroid of $\M(2, 6 \times 5)$ if and only if $\Omega'$ is a base of the matroid of $\M(2, 5 \times 5)$. One easily checks that $\Omega'$ is a relaxed $(2,2,5)$-SLMF and it admits the partition into the following two relaxed $(1,2,5)$-SLMF's, hence $\Omega$ is a base set:
\begin{align*}
\Omega'_{\{1,3,4\}} = 
\begin{bmatrix}
1  & 1 & 0  \\
1  & 1 & 0  \\
1  & 0 & 1  \\
0  & 1 & 1  \\
0  & 0 & 1  
\end{bmatrix}, \, \, \, 
\Omega'_{\{2,5\}} = 
\begin{bmatrix}
 1 & 0 \\
 1  & 0 \\
 0  & 1 \\
1  & 1 \\
 1  & 1 
\end{bmatrix}
\end{align*} 
\end{example}

\bibliographystyle{amsalpha}
\bibliography{DeterminantalMatroid-arxiv-Feb23}

\providecommand{\bysame}{\leavevmode\hbox to3em{\hrulefill}\thinspace}
\providecommand{\MR}{\relax\ifhmode\unskip\space\fi MR }
\providecommand{\MRhref}[2]{%
  \href{http://www.ams.org/mathscinet-getitem?mr=#1}{#2}
}
\providecommand{\href}[2]{#2}
\begin{thebibliography}{CDNG20}

\bibitem[BC03]{bruns2003grobner}
W.~Bruns and A.~Conca, \emph{Gr{\"o}bner bases and determinantal ideals},
  Commutative Algebra, Singularities and Computer Algebra, Springer, 2003,
  pp.~9--66.

\bibitem[BCRV22]{bruns2022determinants}
W.~Bruns, A.~Conca, C.~Raicu, and M.~Varbaro, \emph{{Determinants, Gr{\"o}bner
  Bases and Cohomology}}, Springer, 2022.

\bibitem[Ber17]{bernstein2017completion}
D.~I. Bernstein, \emph{Completion of tree metrics and rank 2 matrices}, Linear
  Algebra and its Applications \textbf{533} (2017), 1--13.

\bibitem[Boo12]{boocher2012free}
A.~Boocher, \emph{Free resolutions and sparse determinantal ideals},
  Mathematical Research Letters \textbf{19} (2012), no.~4, 805--821.

\bibitem[BV88]{BrunsVetter:1988}
W.~Bruns and U.~Vetter, \emph{{D}eterminantal {R}ings}, Springer,
  Berlin/Heidelberg, 1988.

\bibitem[BZ93]{bernstein1993combinatorics}
D.~Bernstein and A.~Zelevinsky, \emph{Combinatorics of maximal minors}, Journal
  of Algebraic Combinatorics \textbf{2} (1993), no.~2, 111--121.

\bibitem[CDNG15]{conca2015universal}
A.~Conca, E.~De~Negri, and E.~Gorla, \emph{Universal {G}r{\"o}bner bases for
  maximal minors}, International Mathematics Research Notices \textbf{2015}
  (2015), no.~11, 3245--3262.

\bibitem[CDNG20]{conca2020universal}
\bysame, \emph{Universal {G}r{\"o}bner bases and {C}artwright--{S}turmfels
  ideals}, International Mathematics Research Notices \textbf{2020} (2020),
  no.~7, 1979--1991.

\bibitem[CLO15]{cox2015ideals}
D.~A. Cox, J.~Little, and D.~O'Shea, \emph{{Ideals, Varieties, and Algorithms:
  An Introduction to Computational Algebraic Geometry and Commutative
  Algebra}}, Springer, 2015.

\bibitem[CMSV18]{conca2018hankel}
A.~Conca, M.~Mostafazadehfard, A.K. Singh, and M.~Varbaro, \emph{Hankel
  determinantal rings have rational singularities}, Advances in Mathematics
  \textbf{335} (2018), 111--129.

\bibitem[Con07]{conca2007linear}
A.~Conca, \emph{Linear spaces, transversal polymatroids and {ASL} domains},
  Journal of Algebraic Combinatorics \textbf{25} (2007), no.~1, 25--41.

\bibitem[CW19]{conca2019lovasz}
A.~Conca and V.~Welker, \emph{{Lov{\'a}sz--Saks--Schrijver ideals and
  coordinate sections of determinantal varieties}}, Algebra \& Number Theory
  \textbf{13} (2019), no.~2, 455--484.

\bibitem[EF65]{edmonds1965transversals}
J.~Edmonds and D.R. Fulkerson, \emph{Transversals and matroid partition},
  Journal of Research of the National Bureau of Standards, B. Mathematics and
  Mathematical Physics \textbf{69B} (1965), 147--153.

\bibitem[Eis88]{eisenbud1988linear}
D.~Eisenbud, \emph{Linear sections of determinantal varieties}, American
  Journal of Mathematics \textbf{110} (1988), no.~3, 541--575.

\bibitem[Eis95]{Eisenbud-2004}
D.~Eisenbud, \emph{Commutative {A}lgebra with a {V}iew {T}oward {A}lgebraic
  {G}eometry}, Springer, 1995.

\bibitem[EN62]{eagon1962ideals}
J.A. Eagon and D.G. Northcott, \emph{Ideals defined by matrices and a certain
  complex associated with them}, Proceedings of the Royal Society of London.
  Series A. Mathematical and Physical Sciences \textbf{269} (1962), no.~1337,
  188--204.

\bibitem[FR15]{fink2015stiefel}
A.~Fink and F.~Rinc{\'o}n, \emph{Stiefel tropical linear spaces}, Journal of
  Combinatorial Theory, Series A \textbf{135} (2015), 291--331.

\bibitem[Fuj05]{fujishige2005submodular}
S.~Fujishige, \emph{Submodular {F}unctions and {O}ptimization}, Elsevier, 2005.

\bibitem[GGR90]{gelfand1990gamma}
I.M. Gelfand, M.I. Graev, and V.S. Retakh, \emph{{$\Gamma$}-series and general
  hypergeometric functions on the manifold of $k \times m$ matrices}, Preprint
  IPM (1990), no.~64.

\bibitem[Har77]{hartshorne1977algebraic}
R.~Hartshorne, \emph{{Algebraic Geometry}}, vol.~52, Springer Science \&
  Business Media, 1977.

\bibitem[HE71]{hochster1971cohen}
M.~Hochster and J.A. Eagon, \emph{{Cohen-Macaulay rings, invariant theory, and
  the generic perfection of determinantal loci}}, American Journal of
  Mathematics \textbf{93} (1971), no.~4, 1020--1058.

\bibitem[Hun83]{huneke1983strongly}
C.~Huneke, \emph{{Strongly Cohen-Macaulay schemes and residual intersections}},
  Transactions of the American Mathematical Society \textbf{277} (1983), no.~2,
  739--763.

\bibitem[KS95]{kalkbrener1995initial}
M.~Kalkbrener and B.~Sturmfels, \emph{Initial complexes of prime ideals},
  Advances in Mathematics \textbf{116} (1995), no.~2, 365--376.

\bibitem[KTT15]{kiraly2015algebraic}
F.~J. Kir{\'a}ly, L.~Theran, and R.~Tomioka, \emph{The algebraic combinatorial
  approach for low-rank matrix completion}, Journal of Machine Learning
  Research \textbf{16} (2015), 1391--1436.

\bibitem[Lan02]{lang2002graduate}
S.~Lang, \emph{Algebra}, Springer, 2002.

\bibitem[Las78]{lascoux1978syzygies}
A.~Lascoux, \emph{Syzygies des vari{\'e}t{\'e}s d{\'e}terminantales}, Advances
  in Mathematics \textbf{30} (1978), no.~3, 202--237.

\bibitem[LS20]{loho2020matching}
G.~Loho and B.~Smith, \emph{Matching fields and lattice points of simplices},
  Advances in Mathematics \textbf{370} (2020), 107232.

\bibitem[Mac16]{macaulay1916algebraic}
F.~S. Macaulay, \emph{The algebraic theory of modular systems}, Cambridge
  Tracts in Mathematics and Mathematical Physics; no. 19. (1916).

\bibitem[MG81]{merle1981sections}
M.~Merle and M.~Giusti, \emph{Sections des vari{\'e}t{\'e}s
  d{\'e}terminantielles par les plans des coordon{\'e}es}, Algebraic Geometry,
  Proceedings, La Rabida, 1981, pp.~103--119.

\bibitem[Nar86]{narasimhan1986irreducibility}
Himanee Narasimhan, \emph{The irreducibility of ladder determinantal
  varieties}, Journal of Algebra \textbf{102} (1986), no.~1, 162--185.

\bibitem[RST20]{rosen2020algebraic}
Z.~Rosen, J.~Sidman, and L.~Theran, \emph{Algebraic matroids in action}, The
  American Mathematical Monthly \textbf{127} (2020), no.~3, 199--216.

\bibitem[RW14]{raicu2014local}
C.~Raicu and J.~Weyman, \emph{Local cohomology with support in generic
  determinantal ideals}, Algebra \& Number Theory \textbf{8} (2014), no.~5,
  1231--1257.

\bibitem[SC10]{singer2010uniqueness}
A.~Singer and M.~Cucuringu, \emph{Uniqueness of low-rank matrix completion by
  rigidity theory}, SIAM Journal on Matrix Analysis and Applications
  \textbf{31} (2010), no.~4, 1621--1641.

\bibitem[Ser12]{serre2012local}
J.-P. Serre, \emph{{Local Algebra}}, Springer Science \& Business Media, 2012.

\bibitem[SS04]{speyer2004tropical}
D.~Speyer and B.~Sturmfels, \emph{{The tropical Grassmannian}}, Advances in
  Geometry \textbf{4} (2004), 389--411.

\bibitem[Stu90]{sturmfels1990grobner}
B.~Sturmfels, \emph{{Gr{\"o}bner bases and Stanley decompositions of
  determinantal rings}}, Mathematische Zeitschrift \textbf{205} (1990), no.~1,
  137--144.

\bibitem[Stu96]{sturmfels1996grobner}
\bysame, \emph{{G}r\"obner bases and convex polytopes}, vol.~8, American
  Mathematical Society, 1996.

\bibitem[SZ93]{sturmfels1993maximal}
B.~Sturmfels and A.~Zelevinsky, \emph{Maximal minors and their leading terms},
  Advances in Mathematics \textbf{98} (1993), no.~1, 65--112.

\bibitem[Vak17]{Vakil-AG}
R.~Vakil, \emph{{The Rising Sea: Foundations of Algebraic Geometry}}, 2017.

\bibitem[Vil01]{villarreal2001monomial}
R.H. Villarreal, \emph{{Monomial Algebras}}, vol. 238, Marcel Dekker New York,
  2001.

\end{thebibliography}

\end{document}